\documentclass[11pt]{article}
\usepackage{authblk}

\usepackage[utf8]{inputenc}
\usepackage[T1]{fontenc}
\usepackage[utf8]{inputenc}
\usepackage{graphicx}
\usepackage{amsmath}
\usepackage{mathtools}
\usepackage{marginnote}
\usepackage[numbers]{natbib}

\usepackage[normalem]{ulem}

\usepackage{amsmath,amssymb,amsthm}
\usepackage{cases}
\usepackage{paralist}
\usepackage{verbatim}
\usepackage[right]{eurosym} 
\usepackage{url}
\usepackage{enumerate}
\usepackage{enumitem}

\usepackage{mathrsfs}
\usepackage{dsfont}

\usepackage{graphicx}

\usepackage{mathtools}
\usepackage{marginnote}
\usepackage{placeins}
\usepackage{subcaption}

\usepackage{todonotes}

\usepackage[top=2.5cm,bottom=2.5cm,left=2.5cm,right=2.5cm]{geometry}
\usepackage[colorlinks=true, allcolors=blue]{hyperref}

\newtheorem{Def}{Definition}[section]
\theoremstyle{plain}
\newtheorem{Rem}[Def]{Remark}

\newtheorem{Lemma}[Def]{Lemma}
\newtheorem{Thm}[Def]{Theorem}
\newtheorem{Ex}[Def]{Example}
\newtheorem{Prop}[Def]{Proposition}
\newtheorem{Cor}[Def]{Corollary}
\newtheorem{Notation}[Def]{Notation}
\newtheorem{Ass}[Def]{Assumptions}
\newtheorem{Alg}[Def]{Algorithm}
\newtheorem{Prob}[Def]{Control problem}

\title{Feedback approaches for set-point stabilization of interacting particle systems}

\author{Daniel Jannik Happ\footnote{\href{mailto:dhapp@uni-wuppertal.de}{dhapp@uni-wuppertal.de}} , Birgit Jacob\footnote{ \href{mailto:bjacob@uni-wuppertal.de}{bjacob@uni-wuppertal.de}} ,
Claudia Totzeck\footnote{\href{mailto:totzeck@uni-wuppertal.de}{totzeck@uni-wuppertal.de} }}  
\affil{Port-Hamiltonian Institute, School of Mathematics and Natural Sciences, \\ University of Wuppertal, Germany}
\date{September 2026}

\newcommand{\R}[0]{\mathbb R}
\newcommand{\N}[0]{\mathbb N}

\newcommand{\dx}[1]{\frac{\mathrm{d}}{\mathrm{d}#1}}
\usepackage{mleftright}
\newcommand{\norm}[1]{\lVert #1 \rVert}
\newcommand{\abs}[1]{\left| #1 \right|}

\newcommand{\supp}{\mathrm{supp}}

\newcommand{\msolinit}{\mu_0}

\newcommand{\charX}{X}
\newcommand{\charV}{V}
\newcommand{\Psiop}{\Psi}
\newcommand{\Pot}{\mathbb{V}}
\newcommand{\Xeq}{\hat{X}}
\newcommand{\Veq}{\hat{v}}
\newcommand{\range}{\mathrm{ran}}
\newcommand{\B}{B}
\newcommand{\cB}{\mathbb{B}}
\newcommand{\HS}{E}
\newcommand{\ks}{E_0}
\newcommand{\RN}[1]{%
  \textup{\uppercase\expandafter{\romannumeral#1}}%
}
\newcommand{\rn}[1]{%
  \textup{\lowercase\expandafter{\romannumeral#1}}%
}
\newcommand{\cnst}{\mathrm{Const}(\msolinit; \R^d)}
\newcommand{\wps}{Q}

\newcounter{thmpart}[Def]

\newcommand{\thmpart}{%
  \par
  \refstepcounter{thmpart}%
  \noindent\textup{(\roman{thmpart})}\quad
}

\begin{document}

\maketitle

\begin{abstract}
We explore the problem of asymptotically stabilizing a class of interacting particle systems to a prescribed particle configuration with zero velocity. The class of interacting particle systems is motivated by the Cucker--Smale model on the mean-field level. Rather than designing a feedback by manually solving an IDA-PBC matching equation, we derive a feedback using an instantaneous rolling horizon control algorithm. We show that this feedback preserves the port-Hamiltonian structure of the interacting particle system and therefore allows for an interpretation in the IDA-PBC framework. Additionally, the energy balance of the closed-loop Hamiltonian provides an important tool to prove convergence of the particles' velocities and interaction forces in discrete and continuous time. If the control operator is surjective, we additionally prove convergence of the particle positions to the prescribed configuration.
\end{abstract}

\begin{minipage}{0.9\linewidth}
\medskip

\end{minipage}

\section{Introduction}
Interacting particle systems refer to a broad class of mathematical models describing the collective behavior of large groups \cite{naldi2010mathematical}. Besides understanding the internal dynamics of the group, it is often important to influence the behavior of the group according to a prescribed aim. Mathematically, this task can be formulated as a control problem, where the dynamics of the system is influenced by an input that can be modified by the user. Choosing this input such that the prescribed aim is satisfied is the fundamental problem in control theory.
    We consider the set-point stabilization problem of steering a continuous time interacting particle system to a prescribed spatial configuration with zero velocity. The system class is motivated by Cucker--Smale type interacting particle systems \cite{CuckerSmale, CuckerSmalemathem}, which have been widely recognized as mathematical models for collective flocking behavior. Those systems can be formulated as \textit{port-Hamiltonian} systems \cite{vdDphSsurvey, JaTo2024}. For port-Hamiltonian systems, several structure preserving control approaches such as \textit{control by interconnection} together with energy shaping by Casimir functions \cite{CBI} and the \textit{interconnection and damping assignment - passivity based control} (IDA-PBC) method \cite{IDAPBCsurvey, IDAPBC} have already been developed, see also \cite{vdSpHcontrol} for an overview of control for port-Hamiltonian systems. In the IDA-PBC approach, one needs to solve a matching equation to shape the dynamics of the underlying system into a desired port-Hamiltonian closed-loop system. The matching equation has many degrees of freedom as the interconnection matrix, the damping matrix and the closed-loop Hamiltonian function can be chosen freely. The matching equation is sufficiently general in the sense that if the desired set-point can be stabilized by some feedback, then the closed-loop system necessarily has a port-Hamiltonian formulation \cite[Lemma $1$]{IDAPBC} and thus a solution to the matching equation exists. However, the proof is not constructive as the closed-loop Hamiltonian is defined by an inverse Lyapunov function. 
    When solving the matching equation in practice, one often has to reduce the degrees of freedom by (partially) fixing either the energy function, or the interconnection and damping matrix \cite{IDAPBCsurvey}. Solving the matching equation in practice can therefore be a difficult manual process. To the best of our knowledge, for a given control operator, no explicit characterization is available for all realizable port-Hamiltonian closed-loop systems. \\
    In contrast, \textit{model predictive control} (MPC) approaches a set-point stabilization problem from an optimization point of view. Here, the input is chosen such that the system trajectory minimizes a user designed cost function which has a minimum at the point to be stabilized \cite{Gruene, MayneconstrMPC}. As an abstract paradigm, MPC can be applied to a large class of dynamical systems. Therefore, it is not inherently designed to preserve additional structures that a particular system might have, such as a port-Hamiltonian formulation. The instantaneous rolling horizon control method is a discrete time inexact variant of MPC, where a prediction horizon of one time step and an approximate solution to the optimization problem are used \cite{Hinzeflcontr, Hinzeinstcontr}. Based on a chosen cost function and a numerical time integration scheme of the underlying continuous time system, instantaneous rolling horizon control explicitly computes a feedback, which then can be interpreted as a discretized version of a continuous time feedback.
    The purpose of the present article is to demonstrate a connection between the IDA-PBC and the instantaneous rolling horizon control method.
    We show that for a class of interacting particle systems, the instantaneous rolling horizon controller computes a feedback leading to a continuous time port-Hamiltonian closed-loop system. Therefore, the feedback can be interpreted in the IDA-PBC framework. Moreover, the closed loop Hamiltonian serves as a Lyapunov function which is not available in an abstract instantaneous control approach and forms the basis for the asymptotic analysis.  \\
    This article is structured as follows: 
    First, we introduce the Cucker--Smale model as a running example in Section \ref{sec:exCS}. In Section \ref{sec:abstractframe}, we introduce the abstract system class together with the set-point stabilization problem. The subsequent Section \ref{sec:instcontr} is devoted to the derivation of the instantaneous rolling horizon feedback controller, which is then interpreted as an IDA-PBC feedback in Section \ref{sec:IDAPBC}. We turn to the well-posedness and convergence analysis of the closed-loop system in Section \ref{sec:WpandAs}. Finally, those results are applied to the mean-field Cucker--Smale model in the final Section \ref{sec:aplications}.

\section{An abstract Port-Hamiltonian control framework for interacting particle systems}\label{Sec:contrlsystem}

    Before introducing the abstract system class, we start with motivating examples, which we also use as running examples throughout this work.
        \subsection{Running examples}\label{sec:exCS}
         
   \begin{Ex}[Cucker--Smale model on the mean-field level] \label{ex:CSmf}
       Consider a possibly infinite number of interacting particles in $\R^d$. Each particle has an $\R^d$-valued position and velocity coordinate $\charX$ and $\charV$. We assume that at time zero, the mass distribution of the particles is given by a compactly supported Borel probability measure on $\R^d \times \R^d$. Then $\charX$ and $\charV$ evolve according to the Newtonian dynamics \begin{subequations}\label{eq:excharsys}
        \begin{align}
      \dx{t} \charX_t &= \charV_t , \label{eq:excharsysX} \\ 
        \dx{t} \charV_t &= - \Psiop(\charX_t) \charV_t - \nabla \Pot (\charX_t) + \cB u(t), \label{eq:excharsysV} \quad 
        \charX(0) = \charX_0, \ \charV(0) = \charV_0 \, .
        \end{align}
    \end{subequations}
       Here, $\charX$ and $\charV$ take values in the Hilbert space $ L^2(\msolinit; \R^d) = L^2(\R^d \times \R^d ; \mu_0; \R^d) $ of all $\msolinit$-square-integrable $\R^d$-valued functions. Moreover, $\Psiop$ and $\nabla \Pot$ are given by
       \begin{align}
        \Psiop(\charX) \colon L^2(\msolinit; \R^d) &\to L^2(\msolinit; \R^d), \, V \mapsto \Psiop(\charX) \charV, \, \charX \in L^2(\msolinit; \R^d) , \nonumber \\
         (\Psiop(\charX)\charV)(x,v) = \int \psi(&\abs{\charX(x,v) - \charX(x^\prime, v^\prime)}) (\charV(x,v) - \charV(x^\prime, v^\prime)) \, \mathrm{d} \msolinit(x^\prime, v^\prime) \, , \label{eq:defCSPsiop} \\
        \Pot \colon L^2(\msolinit; \R^d) \to \R, \quad \Pot (\charX) &= \frac{1}{2} \int \int \mathcal{V}(\charX(x,v) - \charX(x^\prime, v^\prime) ) \, \mathrm{d} \msolinit(x,v) \, \mathrm{d} \msolinit(x^\prime, v^\prime), \nonumber \\
        (\nabla \Pot (\charX))(x,v) &= \int \nabla \mathcal{V}(\charX(x,v) - \charX(x^\prime, v^\prime) ) \, \mathrm{d} \msolinit(x^\prime, v^\prime) \, . \label{eq:defCSPot}
    \end{align}
        In addition, $U$ is an abstract Hilbert space and $\cB \in \mathscr{L}(U;L^2(\msolinit; \R^d))$ is a control operator. The first term on the right-hand side of \eqref{eq:excharsys} was introduced by Cucker and Smale in \cite{CuckerSmale} to model the alignment of particles to a common velocity. The function $\psi \in C(\R_+;\R_+)$ models the alignment strength while $\mathcal{V} \in C^1(\R^d; \R)$ denotes the potential of the binary interactions. According to Newton's third law, $\nabla \mathcal{V}$ has to be antisymmetric.
        We refer to \cite{JaTo2024} for details on the mean-field limit of the system considered here and \cite{golse} for a general introduction to mean-field limits. 
        A system with $N \in \N$ particles is obtained by taking the mass distribution
        $\msolinit = \frac{1}{N} \sum_{j= 1}^N \delta_{\hat{x}_j} \otimes \delta_{\hat{v}_j }$, 
    where $\hat{x}_j, \hat{v}_j, j = 1, \ldots N$ are the particles' initial positions and velocities.
   \end{Ex}
   \begin{Ex}[Cucker--Smale model on the mean-field level with an external control agent]\label{ex:CsextAg}
       The mean-field Cucker--Smale model can be modified to incorporate linear interactions with an external control agent. Let $\B \in \R^{d \times d}$ be symmetric, positive semidefinite and $U \coloneq \R^d$. We assume that a particle located at $x \in \R^d$ experiences an acceleration of $B(u-x)$ towards the position $u$ of the external agent. In this situation, we replace the potential $\Pot  \colon L^2(\msolinit; \R^d ) \to \R$ by  
       \begin{align}
           \Pot (\charX) &= \frac{1}{2} \int \int \mathcal{V}(\charX(x,v) - \charX(x^\prime, v^\prime) ) \, \mathrm{d} \msolinit(x,v) \, \mathrm{d} \msolinit(x^\prime, v^\prime)  + \frac{1}{2} \int  \charX(x,v) \cdot  B \charX(x,v) \, \mathrm{d} \msolinit(x,v) \label{eq:PotextAg},
       \end{align}
       and the control operator is given by 
       \begin{align}
           \cB \colon \R^d \to L^2(\msolinit; \R^d), \ u \mapsto (B u) \mathds{1}. \label{eq:cbextAg}
       \end{align}
       Here, for $v \in \R^d$, $v \mathds{1}$ denotes the constant function in $L^2(\msolinit; \R^d)$ with value $v$, that is 
    \begin{align*}
        (v \mathds{1}) (y,w) = v \text{ for } \msolinit  \text{-almost every } (y,w) \in \R^d \times \R^d \, .
    \end{align*}
   \end{Ex}
   For the examples introduced above, we want to asymptotically stabilize a reference velocity and spatial configuration in the following sense: 
   \begin{Prob}\label{Prob:exstab}
       Given a desired reference velocity $\Veq \in C^1([0,\infty);\R^d)$ and a spatial particle configuration $\Xeq \in L^2(\msolinit; \R^d)$, we want to find a control input $u \colon [0,\infty) \to U$ such that
    \begin{align}
        \charV_t - \Veq(t) \mathds{1} \to 0 \text{ and } \charX_t - \int_0^t \Veq(s) \mathds{1} \, \mathrm{d}s \to \Xeq \text{ in } L^2(\msolinit; \R^d) \text{ as } t \to \infty  \, . \label{eq:convaim}
    \end{align}
   \end{Prob}
   We close this subsection by showing that under additional assumptions on $\cB$, the control Problem \ref{Prob:exstab} in Example \ref{ex:CSmf} can be reduced to the case $\Veq = 0$.
         
   \begin{Lemma}\label{lemma:rightinverse}
       Let $\Theta, U$ be Hilbert spaces and $\cB \in \mathscr{L}(U;\Theta)$ have closed range. Then there exists a right inverse $R$ for $\cB$, that is $R \in \mathscr{L}(\range(\cB) ; U )$ satisfying $\cB R = \mathrm{id}_{\range(\cB)}$.
   \end{Lemma}
   Lemma \ref{lemma:rightinverse} is a consequence of the open mapping theorem, see \cite[Theorem $2.12$]{Brezis}.
    
   \begin{Prop}\label{prop:redctv0}
       Consider the mean-field Cucker--Smale model in Example \ref{ex:CSmf} or  \ref{ex:CsextAg} and assume that the control operator $\cB$ has closed range and satisfies 
       \begin{align}
             \{ \Veq (t) \mathds{1} \mid t \geq 0 \} \subseteq \range(\cB) \, . \label{eq:constfcts}
        \end{align}
       Let $u \in C([0,\infty) ; U )$ and $\charX, \charV \in C^1([0,\infty) ; L^2(\msolinit; \R^d) )$ be solutions to \eqref{eq:excharsysX} and \eqref{eq:excharsysV}, respectively. Let $R$ be a right inverse of $\cB$, $\Veq \in C^1([0,\infty);\R^d)$ and define 
    \begin{align*}
        \tilde{\charX}_t \coloneq \charX_t - \int_0^t \Veq(s) \mathds{1} \, \mathrm{d}s, \quad \tilde{\charV}_t \coloneq \charV_t - \Veq (t) \mathds{1}, \quad \tilde{u}(t) \coloneq u(t) - R  \Big( \dx{t} \Veq (t) \mathds{1} \Big), \quad \, t \geq 0 \, .
    \end{align*}
    Then $\tilde{\charX}, \tilde{\charV} \in C^1([0,\infty);L^2(\msolinit; \R^d))$ satisfy \eqref{eq:excharsysX} and \eqref{eq:excharsysV}, respectively. Similarly, for the Cucker--Smale model with an external agent as in Example \ref{ex:CsextAg}, $\tilde{\charX}, \tilde{\charV}$ satisfy \eqref{eq:excharsysX} and \eqref{eq:excharsysV}, respectively, if we replace $\tilde{u}$ by 
    \begin{align}
        \tilde{u}(t) = u(t) - R  \Big( \dx{t} \Veq (t) \mathds{1} \Big) - \int_0^t \Veq (s) \, \mathrm{d}s \, . \label{eq:tildeuextAg}
    \end{align}
   \end{Prop}
   \begin{proof}
       Assumption \eqref{eq:constfcts} implies that 
       \begin{align}
           M \coloneq \mathrm{span} \{ \Veq (t) \mathds{1} \mid t \geq 0 \} \subseteq \range (\cB) \, . \label{eq:Mdef}
       \end{align}
       Moreover, $M$ is finite dimensional and therefore closed. In particular, $\dx{t} \Veq \mathds{1}$ takes values in $M$ and thus $\tilde{u}$ is well-defined. For the Cuker--Smale model in Example \ref{ex:CSmf}, the fact that $\tilde{\charX}$ and $ \tilde{\charV}$ satisfy \eqref{eq:excharsysX} and \eqref{eq:excharsysV} follows from the translational invariance of $\Psiop$ and $\nabla \Pot$, that is, we have
       \begin{align*}
           \nabla \Pot(\charX + x\mathds{1} ) = \nabla \Pot (\charX), \ \Psiop(\charX + x \mathds{1}) = \Psiop(\charX) , \ \Psiop(\charX) v \mathds{1} = 0 
       \end{align*}
       for all $\charX \in L^2(\msolinit; \R^d)$ and $x,v \in \R^d$. Similarly, if we consider Example \ref{ex:CsextAg} where $\nabla \Pot, \cB$ are given by \eqref{eq:PotextAg} and \eqref{eq:cbextAg}, we use the identity 
       \begin{align*}
           \nabla \Pot(\charX + x \mathds{1} ) = \nabla \Pot(\charX) + (B x) \mathds{1} 
       \end{align*}
       to conclude that $\tilde{\charX}, \tilde{\charV}$ and $\tilde{u}$ as in \eqref{eq:tildeuextAg} satisfy \eqref{eq:excharsysX} and \eqref{eq:excharsysV}.
   \end{proof}
   According to Proposition \ref{prop:redctv0}, the control problem \ref{Prob:exstab} reduces to design $\tilde{u}$ such that $\tilde{X}_t \stackrel{t \to \infty}{\to} \Xeq$ and $\tilde{\charV}_t \stackrel{t \to \infty}{\to} 0$ in $L^2(\msolinit)$.
   \begin{Rem}
       The reduction step to $\Veq = 0$ can also be done without the assumption that $\range(\cB)$ is closed. Indeed, it suffices to construct a right inverse $R$ of $\cB$ on the finite-dimensional space $M $ as in \eqref{eq:Mdef}.
   \end{Rem}

   \subsection{Abstract framework}\label{sec:abstractframe}
    Motivated by the previous examples, we study the following abstract interacting particle system:
    \begin{subequations}\label{eq:controlChar}
    \begin{align}
        \dx{t} \charX_{t} &= \charV_{t} \label{eq:controlCharX}, \\ 
        \dx{t} \charV_{t} &= - \Psiop(\charX_{t}) \charV_{t} - \nabla \Pot (\charX_{t}) + \cB u(t), \label{eq:controlCharV} \quad 
        \charX(0) = \charX_{0}, \quad \charV(0) = \charV_{0} \, ,
    \end{align}
    \end{subequations}
    where $\charX, \charV$ take values in a Hilbert space $\HS$. Subsequently, we always work under the following standing assumptions:
    \begin{Ass}\label{ass:control}
    \begin{enumerate}[leftmargin=2em]
        \item $\HS, U$ are real Hilbert spaces. By $\norm{\cdot}$, we denote the norm of the respective Hilbert space, while scalar products are denoted by $\langle \cdot, \cdot \rangle$.
       \item The potential $\Pot \in C^1(\HS ;\R)$ is bounded from below by a constant $\underline{\mathcal{V}} \in \R$ and there exists a constant $C > 0$ such that we have 
       \begin{align}
           \norm{\nabla \Pot (\charX)}^2 \leq C(1+ \Pot(\charX) - \underline{\mathcal{V}}) \text{ for all } \charX \in \HS \, . \label{eq:potgrwthbd}
       \end{align}
       \item $\Psiop \colon \HS \to \mathscr{L}(\HS)$ is bounded and $\Psi(\charX)$ is self-adjoint and positive semidefinite for all $\charX \in \HS$. 
       \item The control operator $\cB \in \mathscr{L}(U; \HS)$ has closed range.
    \end{enumerate}
   \end{Ass}
   \begin{Rem}
       In Section \ref{sec:aplications}, we will present additional conditions on $\psi$ and $\mathcal{V}$ in the Cucker--Smale model, which ensure that Assumption \ref{ass:control} is satisfied.
   \end{Rem}
    
    In the abstract setting, we consider the following control problem:
    \begin{Prob}\label{problem:stabilize}
        Suppose we are given a desired spatial particle configuration $\Xeq \in \HS$. We want to find a control input $u \colon [0,\infty) \to U$ that asymptotically steers the system \eqref{eq:controlChar} to $\Xeq$, that is, we want to ensure that 
    \begin{align}
        \charV_t \to 0 
        \text{ and } 
        \charX_t  \to \Xeq \text{ in } \HS \text{ as } t \to \infty \, . \label{eq:abstrconvaim}
    \end{align}
    \end{Prob} 
   We close this section with a brief review of the IDA-PBC control approach. 
   To tackle the control problem \ref{problem:stabilize}, the IDA-PBC approach exploits the port-Hamiltonian formulation
   \begin{align*}
            \dx{t} \begin{pmatrix}
                \charX_t \\ 
                \charV_t 
            \end{pmatrix}
             &= \left( 
            \begin{pmatrix}
                0 & I \\
                -I & 0
            \end{pmatrix}
            - \begin{pmatrix}
                0 & 0\\
                0&  \Psiop(\charX_t)  
            \end{pmatrix}
             \right) 
            \nabla H(\charX_t, \charV_t) 
            + \begin{pmatrix}
                0 \\
                \cB 
            \end{pmatrix} u(t), 
            \end{align*}
        with initial condition $\charX(0) = \charX_{0}$ and $ \charV(0) = \charV_{0}$ of the open loop system \eqref{eq:controlChar}. Here, the \textit{Hamiltonian} $H$ is given by
    \begin{align*}
        H \colon \HS \times \HS \to \R , \ 
        H(\charX, \charV) = \frac{1}{2} \norm{\charV}^2 + \Pot(\charX) \, .
    \end{align*}
    The aim of the IDA-PBC approach is to find a feedback law $u(t) = \alpha(\charX_t, \charV_t)$ such that the closed-loop system with initial condition $\charX(0)  = \charX_0$ and $\charV(0) = \charV_0 $ exhibits the asymptotic behavior \eqref{eq:abstrconvaim} and additionally admits a port-Hamiltonian representation
        \begin{align}
            \dx{t} \begin{pmatrix}
                \charX_t \\ 
                \charV_t 
            \end{pmatrix}
            &= (J_d(\charX_t, \charV_t) - R_d(\charX_t, \charV_t) ) \nabla H_{\text{cl}}(\charX_t, \charV_t)  
            \, .  \label{eq:abstclsys}
            \end{align}
        The key difficulty is to design the (skew-symmetric, respectively positive semidefinite) operator-valued maps $J_d, R_d$ and the closed-loop Hamiltonian $H_{\text{cl}}$ in a way that both requirements are fulfilled.

\section{Instantaneous rolling horizon control}\label{sec:instcontr}

    Having introduced the abstract Control problem \ref{problem:stabilize}, we will now turn to solution approaches. Instead of following the IDA-PBC approach directly, we introduce the instantaneous rolling horizon control method in this section. Let $\Xeq \in \HS$ be a fixed desired configuration and let $\Pot_{\text{des}} \in C^1(\HS;\R)$ have a unique global minimum at $\Xeq$. We define the desired Hamiltonian as 
    \begin{align*}
        H_{\text{des}} \colon \HS \times \HS \to \R , \, (\charX,\charV) \mapsto \frac{1}{2} \norm{\charV}^2 + \Pot_{\text{des}} (\charX) \, .
    \end{align*}
    In the following, we use $H_{\text{des}}$ as a cost function for a model predictive control approach. In continuous time, this yields: Given a current time point $t_l$, a prediction horizon $\Delta t > 0$ and the current states $\charX_{t_l}$ and $\charV_{t_l}$, we want to design a control $u\colon [t_l, t_l + \Delta t] \to U$ such that the cost 
    \begin{align}
        J_c(\charX,\charV,u) = \int_{t_l}^{t_l + \Delta t} H_{\text{des}}( \charX_s , \charV_s ) \, \mathrm{d}s \label{eq:Jcconttime}
    \end{align}
    is minimized subject to \eqref{eq:controlChar}. Then the process is repeated on the next time interval $[t_l + \Delta t, t_l + 2 \Delta t]$. We summarize the approach in the following algorithm:
    \begin{Alg}[Continuous time model predictive control]\label{Alg:MPC}
        Given initial conditions $\charX_0, \charV_0 \in \HS$ at time $t_0 = 0$, set $l = 0$ and iterate as follows:
        \begin{enumerate}[leftmargin=2em]
            \item\label{item:MPC1} Compute a minimizer $u^l \colon [t_l, t_l + \Delta t] \to U$ of the cost-functional \eqref{eq:Jcconttime} subject to the dynamics \eqref{eq:controlChar} (provided a minimizer exists in an appropriate space).
            \item Compute $\charX_{t_l + \Delta t}, \charV_{t_l + \Delta t}$ using a minimizer $u^l$, set $t_{l+1} = t_l + \Delta t, \, l = l + 1$ and go back to step \ref{item:MPC1}. 
        \end{enumerate}
    \end{Alg}
    However, for our purposes, the general model predictive control algorithm poses two main difficulties: First, since the system \eqref{eq:controlChar} is in general nonlinear and the cost-functional $J_c$ depends on $u$ only via the state variables $\charX$ and $\charV$, the minimization problem \eqref{eq:Jcconttime} lacks an apparent coercive or convex structure even if $\Pot_{\text{des}}$ is quadratic. Second, even if a minimizer $u$ exists, it is not clear whether it can be represented by a state feedback $u(t) = \alpha (\charX_t, \charV_t)$. However, such a state feedback form is necessary to interpret the closed-loop system in the IDA-PBC framework. For those prescribed reasons, we will use the instantaneous rolling horizon control method instead, which is an inexact version of Algorithm \ref{Alg:MPC} and works as follows \cite{Hinzeflcontr}: 
    On each time-interval $[t_l, t_l + \Delta t]$, the continuous time system \eqref{eq:controlChar} is discretized using some numerical time integration scheme with stepsize $\Delta t$, that is, only a single time step is performed on each subinterval. In short, we write $(\charX^{l+1}, \charV^{l+1} ) = \Phi (\Delta t, \charX^l , \charV^l, u^l)$, where the function $\Phi$ is defined in terms of the chosen numerical integration scheme. Moreover, we replace $J_c$ by a discrete analogue given by
    \begin{align}
        J_{c,d} \colon \HS \times \HS \times U \to \R, \ J_{c,d} (\charX^{l+1},\charV^{l+1},u^l) \coloneq  \Delta t H_{\text{des}}(\charX^{l+1},\charV^{l+1}) \,  . \label{eq:Jcd}
    \end{align}
    We define the reduced discrete cost-functional $f_d^l$ as 
    \begin{align*}
        f_d^l \colon U \to \R , \ f_d^l (u^l) \coloneq J_{c,d} (\Phi(\Delta t, \charX^l, \charV^l, u^l), u^l) \, .
    \end{align*}
     Instead of looking for a global minimum of $f_d^l$, only a single gradient step is computed in each instance of the algorithm. The resulting input is then used to compute $\charX^{l+1}$ and $\charV^{l+1}$. We summarize the approach in the following algorithm:
    \begin{Alg}[Instantaneous rolling horizon control]\label{ICA}
    Given initial conditions $\charX_0, \charV_0$, a numerical integration scheme $\Phi$ with a time step size $\Delta t > 0$, a desired potential $\Pot_{\text{des}} \in C^1(\HS; \R)$, a sequence of initial control guesses $(u_0^l)_{l \in \N_0} \subseteq U$, set $l = 0$ and proceed as follows:
    \begin{enumerate}[leftmargin=2em]\label{Alg:instcontr}
        \item\label{item:ICA1} Compute the updates $(\charX^{l+1}_0, \charV^{l+1}_0) = \Phi(\Delta t, \charX^{l}, \charV^l, u_0^l)$ using the initial guess $u_0^l$.
        \item Compute the gradient $\nabla f_d^l(u_0^l)$ of the reduced discrete cost-functional in terms of $\charX^{l+1}_0$ and $ \charV^{l+1}_0$.
        \item Choose a step-size $\rho^l > 0$ and do a single gradient step
            $u^l = u_0^l - \rho^l \nabla f_d^l (u_0^l)$ .
        \item Compute $(\charX^{l+1},  \charV^{l+1}) = \Phi(\Delta t, \charX^l, \charV^l, u^l)$, set $l = l+1$ and go back to step \ref{item:ICA1}.
    \end{enumerate}
    \end{Alg}
    Typically, we will choose the gradient step-size $\rho^l = \rho > 0$ independent of $l$ and the initial control guess $u_0^l= u_0(\Delta t, \rho, \charX^l, \charV^l)$ as a function of $\charX^{l}, \charV^{l}$ and the step-sizes $\Delta t, \rho $. In this case, the instantaneous control method can be interpreted as a discrete time approximation using the integration scheme $\Phi$ to a closed-loop system with a feedback-law $u^l = \alpha(\Delta t, \rho, \charX^l, \charV^l)$, that is
    \begin{align}
        (\charX^{l+1} , \charV^{l+1}) &= \Phi (\Delta t, \charX^l, \charV^l, \alpha (\Delta t, \rho, \charX^l, \charV^l) ), \label{eq:discrcl}\\
        \text{ where } \alpha (\Delta t, \rho, \charX^l, \charV^l) &=u_0(\Delta t, \rho, \charX^l, \charV^l) - \rho \nabla f_d^l( u_0(\Delta t, \rho, \charX^l, \charV^l)) \label{eq:abstrfeedb} \, .
    \end{align}
    Once the feedback \eqref{eq:abstrfeedb} has been derived, the following questions arise naturally:
    \begin{enumerate}[leftmargin=2em]
        \item If we see $\Delta t, \rho$ as parameters, does the continuous time closed-loop system \eqref{eq:controlChar} with feedback $u(t) = \alpha (\Delta t, \rho, \charX_t, \charV_t)$ admit a port-Hamiltonian formulation?
        \item Does the corresponding continuous time closed-loop system \eqref{eq:controlChar} converge to the desired equilibrium $(\Xeq, 0)$ as $t \to \infty$?
        \item Does the discrete time closed-loop system \eqref{eq:discrcl} converge to the desired configuration $(\Xeq, 0)$ as $l \to \infty$?
    \end{enumerate}
    A key strength of Algorithm \ref{Alg:instcontr} is that the updates $\charX^{l+1}, \charV^{l+1}$ and the feedback law $\alpha$ can be computed explicitly. We will demonstrate this for an Euler-type scheme.
    \begin{Prop}[Instantaneous control for the semi-explicit Euler scheme]\label{prop:semex}
        If the system \eqref{eq:controlChar} is discretized using the semi-explicit Euler scheme
        \begin{subequations}\label{eq:semi-explEuler0}
        \begin{align}
            \charX^{l+1} &= \charX^{l} + \Delta t \charV^{l+1}, \\ 
            \charV^{l+1} &= \charV^{l} - \Delta t \big(\Psiop(\charX^l) \charV^{l} + \nabla \Pot (\charX^l) \big) + \Delta t \cB u^l, \,
            \charX^0 = \charX_{0}, \quad \charV^0 = \charV_{0} \, ,
        \end{align}
        \end{subequations}
        and the cost is given by \eqref{eq:Jcd}, then Algorithm \ref{ICA} computes the following updates:
        \begin{subequations}
            \begin{align}
            \charX^{l+1}_0 &= \charX^{l} + \Delta t \charV_0^{l+1}, \label{eq:semexEulerX0lp} \\ 
            \charV^{l+1}_0 &=   \charV^{l} - \Delta t( \Psiop (\charX^l) \charV^l + \nabla \Pot (\charX^l))  + \Delta t \cB u^l_0 , \label{eq:semexEulerV0lp} \\
            \nabla f_d^l(u_0^l) &= (\Delta t)^2 \cB^* (\charV_0^{l+1} + \Delta t \nabla \Pot_{\text{des}}(\charX_0^{l+1} ) ) , \\
            u^l &= (I - \rho^l (\Delta t)^3 \cB^* \cB)u_0^l - \rho^l (\Delta t)^2 \cB^* V^l + \rho^l (\Delta t)^3 \cB^* (\Psiop(\charX^l) \charV^l + \nabla \Pot(\charX^l) ) \nonumber \\
            &\quad - \rho^l (\Delta t)^3 \cB^* \nabla \Pot_{\text{des}} (X_0^{l+1} )  \label{eq:sem1EuleruUpd} ,  \\
            \charX^{l+1} &= \charX^l + \Delta t \charV^{l+1},  \label{eq:semexEulerXlp} \\ 
            \charV^{l+1} &= \charV^l - \rho^l (\Delta t)^3 \cB \cB^* V^l - \Delta t (I - \rho^l (\Delta t)^3 \cB \cB^* ) (\Psiop(\charX^l) \charV^l + \nabla \Pot(\charX^l) ) \nonumber \\
            &\quad
            + \Delta t (I - \rho^l (\Delta t)^3 \cB \cB^* ) \cB u_0^l   
             - \rho^l (\Delta t)^4 \cB \cB^* \nabla \Pot_{\text{des}}(\charX_0^{l+1} )
            \label{eq:semexEulerVlp} \, .
            \end{align}
        \end{subequations}
    \end{Prop}
    \begin{proof}
        The semi-explicit Euler scheme \eqref{eq:semi-explEuler0} can be written in the form $e(X^{l+1}, V^{l+1}, u^{l}) = 0$, where $e \colon \HS \times \HS \times U \to \HS \times \HS$ is given by 
        \begin{align*}
            e(\charX,\charV,u) = \begin{pmatrix}
                \charX - \charX^l - \Delta t V \\ 
                \charV - \charV^l + \Delta t (\Psiop(\charX^l) \charV^l + \nabla \Pot(\charX^l)) - \Delta t \cB u
            \end{pmatrix}
            \, .
        \end{align*}
        Hence the defining equation for the adjoint state $p^l = (p_\charX^l, p_\charV^l) \in \HS \times \HS$ reads 
        \begin{align*}
            \begin{pmatrix}
                I & -\Delta tI \\
                0 & I 
            \end{pmatrix}^*
            \begin{pmatrix}
                p_X^l \\
                p_V^l
            \end{pmatrix}
            =
            \begin{pmatrix}
                \nabla_X J_{c,d}( \charX^{l+1}_0, \charV^{l+1}_0, u_0^l ) \\
                \nabla_V J_{c,d}( \charX^{l+1}_0, \charV^{l+1}_0, u_0^l )
            \end{pmatrix}
            =
            \Delta t 
            \begin{pmatrix}
                \nabla \Pot_{\text{des}}(\charX^{l+1}_0 ) \\
                \charV^{l+1}_0
            \end{pmatrix} \, ,
        \end{align*}
        where $\charX_0^{l+1}, \charV_0^{l+1}$ are defined by \eqref{eq:semexEulerX0lp} and \eqref{eq:semexEulerV0lp}. Explicitly, we obtain for $p^l$ 
        \begin{align*}
            \begin{pmatrix}
                p_X^l \\
                p_V^l
            \end{pmatrix}
            = \Delta t \begin{pmatrix}
                \nabla \Pot_{\text{des}}(\charX^{l+1}_0)  \\
                V^{l+1}_0 + \Delta t \nabla \Pot_{\text{des}}(\charX^{l+1}_0) 
            \end{pmatrix}  .
        \end{align*}
        Hence, we obtain the gradient identity 
        \begin{align*}
            \nabla f_d^l(u^l_0) = 
            - e_u (\charX_0^{l+1}, \charV_0^{l+1}, u_0^l )^* p^l 
             = (\Delta t)^2 \cB^* (\charV^{l+1}_0 + \Delta t \nabla \Pot_{\text{des}}(\charX^{l+1}_0) ) \, . 
        \end{align*}
        Therefore, a gradient step with stepsize $\rho^l > 0$ yields the control update \eqref{eq:sem1EuleruUpd} and the new position and velocity coordinates \eqref{eq:semexEulerXlp} and \eqref{eq:semexEulerVlp}. 
    \end{proof}

\section{IDA-PBC interpretation of the MPC feedback}\label{sec:IDAPBC}

    Now that we have computed the updates $\charX^{l+1}$ and $\charV^{l+1}$ for the semi-explicit Euler scheme, it remains to design the initial control guess $u_0^l$ in a suitable way to obtain a feedback law $\alpha = \alpha (\Delta t, \rho, \charX^l, \charV^l)$ such that the corresponding continuous time closed-loop system admits a port-Hamiltonian formulation. In the following lemma, we provide a sufficient condition which ensures that the resulting closed-loop system corresponding to a feedback $\alpha$ of the form
    \begin{align}
        \alpha (\charX, \charV) = - c_1 \cB^* \charV + c_2 \cB^* (\Psiop(\charX) \charV + \nabla \Pot(\charX) ) - F(\charX) \label{eq:genIDAPBCfeedback}
    \end{align}
    admits a port-Hamiltonian formulation.
     \begin{Def}
         Let $\Theta$ be a Hilbert space and $S_1,S_2 \in \mathscr{L}(\Theta)$. We write $S_1 \geq S_2$ if $S_1 - S_2$ is self-adjoint and $\langle w, (S_1 - S_2) w \rangle \geq 0$ for all $w \in \Theta$.
     \end{Def}
    \begin{Lemma}\label{lemma:IDAPBC}
        Let $c_1 \geq 0$ and $c_2 \in \R$ with $\abs{c_2} \norm{\cB \cB^*} < 1$, that is $T \coloneq I - c_2 \cB \cB^*$ is invertible, self-adjoint and $T \geq (1 - \abs{c_2} \norm{\cB \cB^*}) I$. Moreover, let $F \in C(\HS; U)$ and assume that there exists $\hat{\Pot} \in C^1(\HS ; \R )$ such that 
        \begin{align}
            \cB F(\charX) = T\nabla \hat{\Pot}(\charX) \text{ for all } \charX \in \HS \, . \label{eq:Gradmatching}
        \end{align}
         We apply the feedback law $u(t) = \alpha(\charX_t, \charV_t)$ to the system \eqref{eq:controlChar}, where $\alpha$ is given by \eqref{eq:genIDAPBCfeedback}.
        Then the closed-loop system reads
        \begin{align}\label{eq:nonPHclosedloop}
            \dx{t} \begin{pmatrix}
                \charX_t \\ 
                \charV_t 
            \end{pmatrix}
            &=
            \begin{pmatrix}
                \charV_t \\
                  - c_1 \cB \cB^* \charV_t - T ( \Psiop(\charX_t) \charV_t + \nabla \Pot(\charX_t) + \nabla \hat{\Pot}(\charX_t) )
            \end{pmatrix} \, , 
            \end{align}
        with initial conditions $ \charX(0) = \charX_{0}, \, \charV(0) = \charV_{0}$.
        Moreover, \eqref{eq:nonPHclosedloop} admits a port-Hamiltonian formulation. Concretely, we define the closed-loop Hamiltonian as 
        \begin{align}
            H_{\text{cl}} \colon \HS \times \HS \to \R , \ H_{\text{cl}}(\charX, \charV) = \frac{1}{2} \langle \charV, T^{-1} \charV \rangle + \Pot(\charX) + \hat{\Pot}(\charX)  \, . \label{eq:genHcl}
        \end{align}
        Then $\charX, \charV \in C^1([0,\infty) ; \HS)$ satisfy \eqref{eq:nonPHclosedloop} if and only if
        \begin{align}
            \dx{t} \begin{pmatrix}
                \charX_t \\ 
                \charV_t 
            \end{pmatrix}
             &= \left( 
            \begin{pmatrix}
                0 & T \\
                -T & 0
            \end{pmatrix}
            - \begin{pmatrix}
                0 & 0\\
                0& T \Psiop(\charX_t) T + c_1 T^{\frac{1}{2}} \cB \cB^* T^{\frac{1}{2}} 
            \end{pmatrix}
             \right) 
            \nabla H_{\text{cl}}(\charX_t, \charV_t) \, , \label{eq:PHclosedloop}
            \end{align}
        with initial conditions $\charX (0) = \charX_{0}$ and $\charV(0) = \charV_{0}$.
    \end{Lemma}
    \begin{proof}
        It is a direct computation using \eqref{eq:Gradmatching} to show that the closed-loop system is given by \eqref{eq:nonPHclosedloop}. In order to obtain \eqref{eq:PHclosedloop} from \eqref{eq:nonPHclosedloop}, we just note that by definition of $T$, $\cB \cB^*$ and $T$ commute. Therefore, also $T^{\frac{1}{2}}$ and $\cB \cB^*$ commute (see \cite[Theorem $2.2$, item (c)]{Conway}), which implies that 
        \begin{align*}
            c_1  T^{\frac{1}{2}} \cB \cB^* T^{\frac{1}{2}} T^{-1} \charV = c_1 \cB \cB^* T T^{-1} \charV = c_1 \cB \cB^* \charV \text{ for all } \charV \in \HS \, . 
        \end{align*}
        Hence, the right-hand sides of \eqref{eq:nonPHclosedloop} and \eqref{eq:PHclosedloop} are equal.
    \end{proof}
    \begin{Rem}
        Condition \eqref{eq:Gradmatching} states that the function $\charX \mapsto T^{-1} \cB F(\charX)$ is a gradient field. Moreover, the stationary points of \eqref{eq:PHclosedloop} are all pairs $(\charX, 0) \in \HS \times \HS$ satisfying $\nabla \Pot(\charX) + \nabla \hat{\Pot}(\charX) = 0$. In particular, $(\Xeq, 0)$ is a stationary point of \eqref{eq:PHclosedloop} if and only if  
            $\nabla \Pot(\Xeq) + \nabla \hat{\Pot} (\Xeq)  = 0$.
    \end{Rem}
    Subsequently, we provide examples of functions $F$ and $\hat{\Pot}$ satisfying condition \eqref{eq:Gradmatching}.
    \begin{Ex}
        Let $c_1, c_2$ be as in Lemma \ref{lemma:IDAPBC}. Clearly, for $F = 0$, we can choose $\hat{\Pot} = 0$. If $F$ is of the form $F(\charX) = c_3 \cB^* (\charX - \Xeq) $ for $c_3 \in \R$, we define 
        \begin{align}
            \hat{\Pot}(\charX) = \frac{c_3}{2} \norm{ \cB^* T^{-\frac{1}{2}}(\charX - \Xeq) }^2 \, . \label{eq:FhatV}
        \end{align}
        Then \eqref{eq:Gradmatching} is satisfied since we have 
        \begin{align*}
            \cB F(\charX)  = c_3 T T^{-1} \cB \cB^* (\charX - \Xeq) 
            = c_3 T T^{-\frac{1}{2} } \cB \cB^* T^{-\frac{1}{2} } (\charX - \Xeq) = T \nabla \hat{\Pot}(\charX) \, .
        \end{align*}
        In the second step, we have used that $T^{-\frac{1}{2} }$ commutes with $\cB \cB^*$.
    \end{Ex}
    In the previous example, $(\Xeq, 0)$ is a stationary point of \eqref{eq:PHclosedloop} if and only if $\nabla \Pot(\Xeq) = 0$. However, if $\nabla \Pot(\Xeq) \neq 0$, then we can enforce stationarity of $(\Xeq, 0)$ as follows: 
    \begin{Ex}
        Let $c_1, c_2$ be as in Lemma \ref{lemma:IDAPBC} and $c_3 \in \R$. Moreover, we assume that $\cB \hat{\alpha} = \nabla \Pot(\Xeq)$ for some $\hat{\alpha} \in U$. Consider $F(\charX) = - \hat{\alpha} + c_2 \cB^* \nabla \Pot (\Xeq)$. Then \eqref{eq:Gradmatching} holds for $\hat{\Pot}(\charX) = - \langle \nabla \Pot(\Xeq) , \charX - \Xeq\rangle $, since 
        \begin{align*}
            \cB F(\charX) = - \nabla \Pot(\Xeq) + c_2 \cB \cB^* \nabla \Pot(\Xeq) = - T \nabla \Pot(\Xeq) = T \nabla \hat{\Pot}(\charX) \, .
        \end{align*}
        Similarly, if $F$ is given by 
             $F(\charX) = c_3 \cB^* (\charX - \Xeq) - \hat{\alpha} + c_2 \cB^* \nabla \Pot (\Xeq)$,
        then we can choose 
        \begin{align}
            \hat{\Pot}(\charX) = \frac{c_3}{2} \norm{ \cB^* T^{-\frac{1}{2}}(\charX - \Xeq) }^2 - \langle \nabla \Pot(\Xeq) , \charX - \Xeq \rangle \, . \label{eq:statadFhatV}
        \end{align}
        Indeed, this yields
        \begin{align*}
            \cB F(\charX) = c_3 T T^{-\frac{1}{2} } \cB \cB^* T^{-\frac{1}{2} } (\charX - \Xeq) 
            - T \nabla \Pot (\Xeq) 
            = T \nabla \hat{\Pot}(\charX) \, .
        \end{align*}
    \end{Ex}    
    Since we have verified the gradient condition \eqref{eq:Gradmatching} in the previous examples, combining Proposition \ref{prop:semex} with Lemma \ref{lemma:IDAPBC} yields the following IDA-PBC interpretations.
        
        \begin{Cor}[Instantaneous control for the semi-explicit Euler scheme and quadratic desired potential]\label{cor:semexquadvdes}
            Let the cost be given by \eqref{eq:Jcd} with a quadratic desired potential $\Pot_{\text{des}} (\charX) = c/{2} \lVert \charX - \Xeq \big\rVert^2$ for some $c \geq 0$. Consider the semi-explicit Euler scheme \eqref{eq:semi-explEuler0},
        fix $\Delta t > 0$ and a constant gradient step size $\rho^l = \rho > 0, \, l \in \N_0$.
        \thmpart Let $u_0^l = 0$ for all $l \in \N_0$. Then the feedback \eqref{eq:abstrfeedb} is of the form \eqref{eq:genIDAPBCfeedback} with
            \begin{align} 
                c_1 &\coloneq \rho (\Delta t)^2 \big(1 +c(\Delta t)^2 \big) , \ c_2 \coloneq \rho (\Delta t)^3 \big(1+c(\Delta t)^2 \big), \
                c_3 \coloneq c \rho (\Delta t)^3, \
                F(\charX) \coloneq c_3 \cB^* (\charX - \Xeq)  , \label{eq:c1c3semex}
            \end{align}
            and therefore Algorithm \ref{Alg:instcontr} computes the semi-explicit Euler iterates with initial condition $\charX (0) = \charX_{0}$, $ \charV (0) = \charV_{0}$ and step-size $\Delta t$ of the continuous time system 
            \begin{subequations}\label{eq:semiexpleulerconttime}
            \begin{align}
                \dx{t} \charX_t &= \charV_t \, , \\
                \dx{t}
                    \charV_t
                 &= 
                 - c_1 \cB \cB^* \charV_t - (I - c_2 \cB \cB^* ) (\Psiop(\charX_t) \charV_t + \nabla \Pot (\charX_t) ) 
                - c_3 \cB \cB^* (\charX_t - \Xeq)  
                .
                \end{align}
            \end{subequations}
            In particular, if $c_2 \norm{\cB \cB^*} < 1$, then Lemma \ref{lemma:IDAPBC} applies with
                \begin{align}
                \hat{\Pot}(\charX) =  \frac{c_3}{2} \norm{ \cB^* T^{-\frac{1}{2}}(\charX - \Xeq) }^2 \, . \label{eq:semiexnonstatad}
                \end{align}
            \thmpart Suppose that $\nabla \Pot (\Xeq) \in \range (\cB)$, that is $\cB \hat{\alpha } = \nabla \Pot (\Xeq)$ for some $\hat{\alpha} \in U$. For $l \in \N_0$, take $u_0^l = \hat{\alpha}$. Then the feedback \eqref{eq:abstrfeedb} is of the form \eqref{eq:genIDAPBCfeedback} with
            \begin{align}
                c_1 &\coloneq \rho (\Delta t)^2 \big(1+c(\Delta t)^2 \big) , \ c_2 \coloneq \rho (\Delta t)^3 \big(1+c(\Delta t)^2 \big), 
                c_3 \coloneq c \rho (\Delta t)^3, \nonumber \\
                F(\charX) &\coloneq c_3 \cB^* (\charX - \Xeq) - \hat{\alpha} + c_2 \cB^* \nabla \Pot (\Xeq) \label{eq:c1c2c3semex} \, , \nonumber
            \end{align}
            and therefore Algorithm \ref{Alg:instcontr} computes the semi-explicit Euler iterates with initial condition $\charX(0) = \charX_{0}, \ \charV(0) = \charV_{0}$ and step-size $\Delta t$ of the continuous time system             \begin{subequations}\label{eq:semiexpleulerconttimestatad}
            \begin{align}
                \dx{t} \charX_t &= \charV_t, \\
                \dx{t} \charV_t &= - c_1 \cB \cB^* \charV_t - (I - c_2 \cB \cB^* ) (\Psiop(\charX_t) \charV_t + \nabla \Pot (\charX_t) - \nabla \Pot(\Xeq) )  
                - c_3 \cB \cB^* (\charX_t - \Xeq) \, . 
                \end{align}
            \end{subequations}
            In particular, if $c_2 \norm{\cB \cB^*} < 1$, then Lemma \ref{lemma:IDAPBC} applies with 
            \begin{align}
                \hat{\Pot}(\charX) =  \frac{c_3}{2} \norm{ \cB^* T^{-\frac{1}{2}}(\charX - \Xeq) }^2
                - \langle \nabla \Pot(\Xeq), \charX - \Xeq \rangle \, . \label{eq:semiexstatad}
            \end{align}
        \end{Cor}

\section{Well-posedness and asymptotic behavior of the closed-loop system}\label{sec:WpandAs}

    Having established a port-Hamiltonian formulation of the closed-loop system obtained by the instantaneous rolling horizon controller, we turn our attention to the analytical properties of \eqref{eq:PHclosedloop}, namely well-posedness of the continuous time system and large-time asymptotic behavior both in discrete and continuous time. We start with well-posedness. For the Cucker--Smale model in Section \ref{sec:exCS}, the right-hand side of \eqref{eq:nonPHclosedloop} naturally satisfies a local Lipschitz condition with respect to the supremum norm rather than in the Hilbert space $L^2(\msolinit; \R^d)$ (see Proposition \ref{prop:PsiandPotprop} below). In the abstract setting, this motivates to formulate a well-posedness result on a continuously embedded subspace.   
    \begin{Thm}[Well-posedness of the system \eqref{eq:nonPHclosedloop} in a subspace]\label{Thm:wpcl}
        Let $c_1, c_2 \in \R $ and define $T \coloneq I - c_2 \cB \cB^*$. Let $(\wps, \norm{\cdot}_{\wps} )$ be a Banach space that embeds continuously into $\HS$ and suppose that $\wps$ is an invariant subspace of $\cB \cB^*$. For the potential $\hat{\Pot} \in C^1(\HS; \R ) $, we assume that: 
        \thmpart\label{ass:hatV} $T \nabla \hat{\Pot} (\charX) \in \wps$ for all $\charX \in \wps$ and the map 
                    $g \colon \wps \to \wps, \, \charX \mapsto T \nabla \hat{\Pot}(\charX)$
                is locally Lipschitz and satisfies a linear growth condition, that is, there exists $C > 0$ such that $\norm{g(\charX)}_{\wps} \leq C ( 1 + \norm{\charX }_{\wps} )$ for all $\charX \in \wps $.
            \thmpart $\Psiop (\charX)$ restricts to a bounded operator on $\wps$ for all $\charX \in \wps$ and $\Psiop(\cdot) \colon \wps \to \mathscr{L}(\wps)$ is locally Lipschitz and bounded. 
            \thmpart $\nabla \Pot(\charX) \in \wps$ for all $\charX \in \wps$ and $\nabla \Pot( \cdot) \colon \wps \to \wps$ is locally Lipschitz with respect to $\norm{\cdot}_{\wps}$ and there exists $C > 0$ such that $\norm{\nabla \Pot (\charX)}_{\wps} \leq C ( 1 + \norm{\charX }_{\wps} )$ for all $\charX \in \wps$. \par 
        Then for every $\charX_0, \charV_0 \in \wps$, there exists a unique solution $\charX, \charV \in C^1([0,\infty);\wps )$ to the system \eqref{eq:nonPHclosedloop}. Moreover, if $\abs{c_2} \norm{\cB \cB^*}_{\mathscr{L}(\HS )} < 1$, there exists a solution $\charX, \charV \in C^1([0,\infty) ; \HS) $ to the port-Hamiltonian system \eqref{eq:PHclosedloop}.
    \end{Thm}    
    \begin{proof}
        Since $\wps $ is an invariant subspace of $\cB \cB^*$, the restriction $\cB \cB^* \mid_{\wps }$ is a linear map. Moreover, as $\wps $ is a Banach space and $\cB \cB^* \in \mathscr{L}(\HS )$, the closed graph theorem implies that $\cB \cB^* \mid_{\wps } \in \mathscr{L}(\wps)$. Thus, we have 
            $T \mid_{\wps } = I_{\wps } - c_2 \cB \cB^* \mid_{\wps } \in \mathscr{L}(\wps)$.
        Since $\Psiop$, $\nabla \Pot$ and $g$ are locally Lipschitz with respect to $\norm{\cdot}_{\wps}$, the right-hand side of \eqref{eq:nonPHclosedloop} is locally Lipschitz with respect to the norm on $\wps$. Similarly,  the right-hand side of \eqref{eq:nonPHclosedloop} satisfies a linear growth condition with respect to $\norm{\cdot}_{\wps}$. Hence, global existence and uniqueness of solutions $\charX, \charV \in C^1([0,\infty);\wps) $ follow from the Picard-Lindelöf theorem and Gronwall's inequality. To prove the addendum, we assume that $\abs{c_2} \norm{\cB \cB^*}_{\mathscr{L}(\HS )} < 1$. Then $T$ has a bounded inverse in $\mathscr{L}(\HS)$ and therefore the system \eqref{eq:PHclosedloop} is well-defined. Since the embedding $\wps \hookrightarrow \HS $ is continuous, we have $\charX, \charV \in C^1([0,\infty);\wps ) \subseteq C^1([0,\infty);\HS)$. According to Lemma \ref{lemma:IDAPBC}, $\charX, \charV$ are solutions to \eqref{eq:PHclosedloop}.
    \end{proof}
    \begin{Rem}
        Let the assumptions of Theorem \ref{Thm:wpcl} be satisfied. If the system \eqref{eq:PHclosedloop} is discretized using the semi-explicit Euler scheme, then the resulting iterates $\charX^l, \charV^l$ of the discrete time system lie in $\wps$ provided that $\charX_0, \charV_0 \in \wps$.
    \end{Rem}
    \begin{Ex}
        Let $c_2 \in \R$ with $\abs{c_2} \norm{\cB \cB^*}_{\mathscr{L}(\HS )} < 1$, that is $T \coloneq I - c_2 \cB \cB^*$ is invertible in $\mathscr{L}(\HS)$. Moreover, let $\wps$ be invariant under $\cB \cB^*$. Then Assumption \ref{ass:hatV} in Theorem \ref{Thm:wpcl} is satisfied if $\hat{\Pot}$ is given by \eqref{eq:FhatV} and $\Xeq \in \wps$. Indeed, as shown in the proof of Theorem \ref{Thm:wpcl}, $\cB \cB^*$ restricts to a bounded operator on $\wps$ and thus we have for
        \begin{align*}
            T \nabla \hat{\Pot}(\charX) = c_3 T T^{- \frac{1}{2}} \cB \cB^* T^{- \frac{1}{2}} (\charX - \Xeq) = c_3 \cB \cB^* (\charX - \Xeq), \, \charX \in \wps , 
        \end{align*}
        which is Lipschitz and thus has linear growth. If $T \nabla \Pot(\Xeq) \in \wps $ the same argument also applies if $\hat{\Pot}$ is given by \eqref{eq:statadFhatV}. 
    \end{Ex}
    Subsequently, we analyze the asymptotic behavior of solutions to the closed-loop system \eqref{eq:PHclosedloop}. Since the port-Hamiltonian formulation of \eqref{eq:PHclosedloop} uses the Hilbert space structure of $\HS$, we will restrict our attention to convergence in $\HS$ even if the system is well-posed in a continuously embedded subspace $\wps$. Fortunately, the following results only require the existence of a solution $\charX, \charV \in C^1([0,\infty) ;\HS )$. Our analysis is based on {B}arb\u alat's lemma, which we recall here. A proof can be found in \cite[Theorem $4$]{FarkasWegner}.
    \begin{Lemma}[{B}arb\u alat's lemma] \label{lemma:barbalat} 
    Let $\Theta$ be a Banach space and $g \in C^{1}([0, \infty ); \Theta)$ such that the limit $\lim\limits_{t \to \infty} g(t) = g_\infty$ exists in $\Theta$. If $g^\prime \in C([0,\infty) ; \Theta)$ is uniformly continuous, then $ g^\prime (t) \to  0 $ in $\Theta$ as $t \to \infty$.
    \end{Lemma}
    \begin{Thm} \label{thm:BarbConv}
        Let $c_1,c_2, F, \hat{\Pot}$ be as in Lemma \ref{lemma:IDAPBC} and $\Xeq \in \HS$. Let $\charX, \, \charV \in C^1([0,\infty) ; \HS )$ be solutions to the system \eqref{eq:PHclosedloop}. Consider the following assertions:
        \begin{enumerate}[label=(\roman*), leftmargin=2em]
            \item \label{ass:feedlowbound} The potential $ \Pot + \hat{\Pot} \in C^1(\HS ; \R) $ is bounded from below by a constant $\underline{\hat{\Pot}} \in \R$. 
            \item\label{ass:feedunifcont} $\nabla \Pot + \nabla \hat{\Pot} \in  C(\HS ; \HS )$ is uniformly continuous and the family $ (\nabla \Pot (\charX_t) + \nabla \hat{\Pot}(\charX_t) )_{t \geq 0} \subseteq \HS $ is bounded. 
            \item\label{ass:feedposdef} There exists $\nu > 0$ such that for all $Y\in \HS $, we have 
            \begin{align*}
                \Psiop(Y) + c_1 T^{- \frac{1}{2}} \cB \cB^* T^{- \frac{1}{2}} \geq \nu I \, .
            \end{align*}
            \item\label{ass:feedposconv} For every $\varepsilon > 0$, there exists $\delta > 0$ such that for all $\charX \in \HS$, we have 
            \begin{align}
                \norm{ \nabla \Pot(\charX) + \nabla \hat{\Pot}(\charX) } \leq \delta \ \Rightarrow \  \norm{\charX - \Xeq } \leq \varepsilon \, . \label{eq:Forceimpl}
            \end{align}
        \end{enumerate}
         If Assertions \ref{ass:feedlowbound}, \ref{ass:feedunifcont} and \ref{ass:feedposdef} are satisfied, then $\charV_t \to 0$ and $\nabla \Pot (\charX_t) + \nabla \hat{\Pot}(\charX_t) \to 0 $ in $\HS$ as $t \to \infty $. If additionally \ref{ass:feedposconv} is satisfied, then $\charX_t \to \Xeq$ in $\HS $ as $t \to \infty$. 
    \end{Thm}
    \begin{proof}
        The strategy of the proof is similar to \cite{JaTo2024}. For the first part, let us assume that \ref{ass:feedlowbound}, \ref{ass:feedunifcont} and \ref{ass:feedposdef} are satisfied. From the port-Hamiltonian formulation of the closed-loop system \eqref{eq:PHclosedloop}, we see that the energy balance reads 
        \begin{align*}
            \dx{t} H_{\text{cl}} (\charX_t, \charV_t) = - \langle \charV_t, \big(\Psiop(\charX_t) + c_1 T^{-\frac{1}{2}} \cB \cB^* T^{-\frac{1}{2}} \big) \charV_t \rangle \leq - \nu \norm{\charV_t}^2 \overset{\ref{ass:feedposdef}}{\leq} 0 \, .
        \end{align*}
        Since $\Pot + \hat{\Pot}$ is bounded from below by a constant $\underline{\hat{\Pot}}$, we see that $H_{\text{cl}}(\charX, \charV) \geq \underline{\hat{\Pot}}$ for all $\charX, \charV \in \HS$. Hence, we conclude that the monotone limit 
           $ H_\infty \coloneq \lim_{t \to \infty} H_{\text{cl}}(\charX_t, \charV_t)$ 
        exists and 
        \begin{align}
            \int_0^\infty \norm{\charV_t}^2 \, \mathrm{d} t \leq \frac{1}{\nu} \big( H_{\text{cl}}(\charX_0, \charV_0) - H_{\infty} ) < \infty \, ,
        \end{align}
        that is, $\charV \in L^2([0,\infty) ; \HS )$. Moreover, since $H_{\text{cl}}$ is decreasing, we have for all $t \geq 0$
        \begin{align*}
            \norm{\charV_t}^2 \leq \norm{ T^\frac{1}{2}}^2 \langle \charV_t , T^{-1} \charV_t \rangle \leq 2 \norm{T} \big( H_{\text{cl}}(\charX_t, \charV_t) - \underline{\hat{\Pot}} \big) 
            \leq 2\norm{T} \big( H_{\text{cl}}(\charX_0, \charV_0) - \underline{\hat{\Pot}} \big) \, ,
        \end{align*}
        which shows that $(\charV_t)_{t \geq 0} \subseteq \HS$ is bounded. Therefore, $(\charX_t)_{t \geq 0}$ is Lipschitz. Moreover, since we have assumed that $(\nabla \Pot (\charX_t) + \nabla \hat{\Pot}(\charX_t) )_{t \geq 0} \subseteq \HS$ is bounded, we obtain that $(\norm{\charV_t}^2)_{t \geq 0}$ is Lipschitz. Indeed, we have 
        \begin{align*}
            \dx{t} \norm{\charV_t}^2 = 2 \langle \charV_t, - T(\nabla \Pot (\charX_t) + \nabla \hat{\Pot}(\charX_t) ) - (T \Psiop(\charX_t) + c_1 \cB \cB^* ) \charV_t \rangle 
        \end{align*}
        and the right-hand side is bounded since $\charV$ and $\Psiop$ are bounded. {B}arb\u alat's lemma~\ref{lemma:barbalat} implies that $\charV_t \stackrel{t \to \infty}{\longrightarrow} 0$ in $\HS$. As a next step, we claim that $\dx{t}\charV \in C([0,\infty ) ; \HS )$ is uniformly continuous. Indeed, we have 
        \begin{align}
            \dx{t} \charV_t &= - (T \Psiop(\charX_t) + c_1 \cB \cB^* ) \charV_t - T (\nabla \Pot(\charX_t) + \nabla \hat{\Pot}(\charX_t) ) 
            \eqcolon D(t) - T (\nabla \Pot(\charX_t) + \nabla \hat{\Pot}(\charX_t) ) \label{eq:dxvt} \, .
        \end{align}
        Note that continuity of $\charX $ and $\dx{t} \charV$ imply $D$ is continuous. Moreover, convergence of $\charV$ implies that $D$ converges to zero as $t \to \infty$ and thus $D$ is uniformly continuous. For the second term on the right-hand side of $\eqref{eq:dxvt}$, we use that $\nabla \Pot + \nabla \hat{\Pot}$ is uniformly continuous. Thus, for $\varepsilon > 0$, there exists $\delta > 0$ such that 
        \begin{align*}
            \norm{\nabla \Pot (Y) + \nabla \hat{\Pot}(Y) - (\nabla \Pot (Z) + \nabla \hat{\Pot}(Z) ) } < \varepsilon \, \text{whenever} \, Y, Z \in \HS \text{ with } \norm{Y - Z } < \delta .
        \end{align*}
        Now, for $s,t \geq 0$ with $\abs{t-s} < \frac{\delta}{\norm{\charV}_{\infty } } $, Lipschitz continuity of $\charX$ yields 
        \begin{align*}
            \norm{\charX_t - \charX_s} \leq \norm{\charV}_{\infty } \abs{t-s} < \delta , \, 
            \text{thus}  \norm{\nabla \Pot (\charX_t ) + \nabla \hat{\Pot}(\charX_t) - (\nabla \Pot (\charX_s) + \nabla \hat{\Pot}(\charX_s) ) } < \varepsilon  .
        \end{align*}
        Thus, we have shown that $\dx{t} \charV$ is uniformly continuous. Since $\charV_t \to 0$ as $t \to \infty$, another application of {B}arb\u alat's lemma \ref{lemma:barbalat} implies that 
        \begin{align*}
            0 = \lim\limits_{t \to \infty} \dx{t} \charV_t = - \lim\limits_{t \to \infty}  T (\nabla \Pot(\charX_t) + \nabla \hat{\Pot}(\charX_t) ) \, ,
        \end{align*}
        where the above limit is understood in $\HS$. Since $T$ has a bounded inverse, we conclude that $ \nabla \Pot(\charX_t) + \nabla \hat{\Pot}(\charX_t)  \stackrel{t \to \infty}{\to} 0$.
        To complete the proof, let us additionally assume that Assertion \ref{ass:feedposconv} is satisfied and fix $\varepsilon > 0$. Then there exists $\delta > 0$ such that the implication \eqref{eq:Forceimpl} holds. Moreover, we can find $t_0 > 0$ such that for all $t \geq t_0$, we have 
        \begin{align*}
            \norm{\nabla \Pot(\charX_t) + \nabla \hat{\Pot}(\charX_t)} \leq \delta \text{ and therefore } \norm{\charX_t - \Xeq} \leq \varepsilon \, . \qquad \qedhere
        \end{align*}
    \end{proof}
    \begin{Notation}
        Let $\Theta$ be a Hilbert space and $S \in \mathscr{L}(\Theta)$ self-adjoint. For a Borel-measurable subset $A \subseteq \R$, we write $\mathds{1}_A(S) \in \mathscr{L}(\Theta)$ for the spectral measure of $A$ associated to $S$, see \cite[Theorem $2.2$]{Conway}. In particular, for $\lambda \in \R$, $\mathds{1}_{\{ \lambda \} } (S)$ is the orthogonal projection onto the closed subspace $\ker(\lambda I -  S) \subseteq \Theta$, see \cite[Theorem $12.29$]{RudinFA}. 
    \end{Notation}
    \begin{Prop}\label{Prop:lowbound}
        Let $\Theta_1, \Theta_2, \Theta_3$ be Hilbert spaces, $M \in \mathscr{L}(\Theta_1; \Theta_3)$ with $M \neq 0$ and $ D \in \mathscr{L}(\Theta_2; \Theta_1) $. Let $S \in \mathscr{L}(\Theta_1)$ be an isomorphism that commutes with $D D^*$. Moreover, assume that $\range(D) \subseteq \Theta_1$ is closed and $\range(M^*) \subseteq \range(D)$. Then there exists a constant $\kappa > 0$ depending only on $D$ such that 
        \begin{align}
            \norm{D^* S w}_{\Theta_2} \geq \frac{\kappa}{\norm{S^{-1}} \norm{M}} \norm{M w}_{\Theta_3} \text{ for all } w \in \Theta_1 \, . \label{eq:DSlowbound}
        \end{align}
    \end{Prop}
    \begin{Rem}
        In the setting of Proposition \ref{Prop:lowbound}, Douglas range inclusion theorem \cite[Theorem $1$]{Douglasrangeincl} implies that there exists a constant $\kappa > 0$ depending on $M$ and $D$ such that 
            $\norm{D^* w}_{\Theta_2} \geq \kappa \norm{Mw}_{\Theta_3}$.
        The additional assumption that $S$ commutes with $DD^*$ ensures that we can replace $w$ by $Sw$ on the left hand side and that we get a more explicit dependence of $\kappa$ on $S$ and $M$.
    \end{Rem}
    \begin{proof}[Proof of Proposition \ref{Prop:lowbound}]
        We start by noting that the Hilbert space $\Theta_1$ splits as 
        \begin{align*}
            \Theta_1 = \ker(D^*) \oplus \ker(D^*)^\perp = \ker(D^* ) \oplus \overline{\range(D)} = \ker (D^*) \oplus \range (D) \, .
        \end{align*}
        Therefore, $D^* \! \mid_{\range(D)}$ is injective and $\range (D^* \! \mid_{\range(D)}) = \range(D^*)$ is closed by the closed range theorem. The open mapping theorem implies that $Q  \coloneq (D^* \! \mid_{\range(D)})^{-1}$ is a bounded operator from $\range(D^* )$ to $\range(D)$. Let $P \in \mathscr{L}(\Theta_1)$ denote the orthogonal projection onto the closed subspace $\range (D) \subseteq \Theta_1$. Then $I - P$ is the orthogonal projection onto $\ker(D^*) = \ker(DD^*)$ and therefore, we have 
            $I - P = \mathds{1}_{ \{ 0 \} } (D D^* )$,
        where the right-hand side is defined via the functional calculus of the bounded self-adjoint operator $DD^*$. In particular,
          $  P = I -\mathds{1}_{ \{ 0 \} } (D D^* ) = \mathds{1}_{(0,\infty)}(DD^*) $. 
        Since $S$ commutes with $DD^*$, $S$ also commutes with $P$. Therefore, for all $w \in \Theta_1$, we have
        \begin{align*}
            D^* S (I - P) w = D^* (I-P) S w = 0 \, ,
        \end{align*}
        where we have used that $I-P$ is the orthogonal projection onto $\ker(D^*)$. Hence,
        \begin{align*}
            \norm{D^* S w }_{\Theta_2} &= \norm{D^* S P w + D^* S (I - P) w}_{\Theta_2} 
            = \norm{D^* S P w}_{\Theta_2} \geq \frac{1}{\norm{Q}} \norm{Q D^*  P S w}_{\Theta_1} \\ 
            &= \frac{1}{\norm{Q}} \norm{P Sw}_{\Theta_1}  
            = \frac{1}{\norm{Q}} \norm{ S P w}_{\Theta_1} 
            \geq \frac{1}{\norm{Q} \norm{S^{-1}}} \norm{Pw}_{\Theta_1} \\
            &\geq \frac{1}{\norm{Q} \norm{S^{-1}} \norm{M}} \norm{M Pw}_{\Theta_3}
            = \frac{1}{\norm{Q} \norm{S^{-1}} \norm{M}} \norm{Mw}_{\Theta_3} \, .
        \end{align*}
        In the last step, we have used that $\ker (D^*) = \range(D)^\perp \subseteq \range(M^*)^\perp = \ker(M)$ and therefore $(I - P)w \in \ker(D^*)\subseteq \ker(M)$ implies that $M w = M P w$. In summary, we have shown that \eqref{eq:DSlowbound} holds with $\kappa \coloneq \norm{Q}^{-1}$.
    \end{proof}
    \begin{Cor}\label{Cor:Blowbound}
        Let $c_2 \in \R $ with $\abs{c_2} \norm{\cB \cB^*} < 1$ and define $T \coloneq I - c_2 \cB \cB^*$. Moreover, let $\ks \subseteq \HS$ be a closed subspace satisfying $E_0 \subseteq \range (\cB)$ and let $P_{\ks} \in \mathscr{L}(\HS)$ be the orthogonal projection onto $\ks$. Then there exists a constant $\kappa > 0$ depending only on $\cB$ such that 
        \begin{align}
            \norm{\cB^* T^{- \frac{1}{2}} W } \geq \frac{\kappa}{\lVert T^{ \frac{1}{2}}\rVert}\norm{ P_{\ks} W} 
            \geq \frac{\kappa}{\sqrt{2} } \norm{P_{\ks} W} \text{ for all } W \in \HS \, .  \label{eq:BsPlowbound}
        \end{align}
        Thus, if $\cB$ is surjective, there exists a constant $\sigma > 0$ depending only on $\cB$ such that 
        \begin{align}
             \norm{\cB^* T^{- \frac{1}{2}} W } \geq \frac{\sigma }{\lVert T^{ \frac{1}{2}}\rVert } \norm{W} 
            \geq \frac{\sigma }{\sqrt{2} } \norm{W} \text{ for all } W \in \HS \, . \label{eq:Blowbound}
        \end{align}
    \end{Cor}
    \begin{proof}
        Note that if $\ks = \{ 0 \}$, then \eqref{eq:BsPlowbound} is trivially satisfied. Thus, we can assume that $\ks$ is nonzero. We define the operator $M \in \mathscr{L}( \HS )$ by $M = P_{\ks}$. Then $\norm{M} = 1$ and $M^* = M$. Therefore $\range(M^*) = \ks   \subseteq \range (B)$. Moreover, we note that $T$ commutes with $\cB \cB^*$ and therefore $T^{-\frac{1}{2}} $ also commutes with $\cB \cB^*$. Hence we can apply Proposition \ref{Prop:lowbound} to the operators $S = T^{-\frac{1}{2}}, D = \cB$ and $M = P_{\ks}$ to obtain the existence of $\kappa > 0$ depending only on $\cB$ such that 
        \begin{align*}
            \norm{\cB^* T^{-\frac{1}{2}} W } \geq \frac{\kappa}{\norm{M} \norm{T^{\frac{1}{2}}}} \norm{P_{\ks} W} \geq \frac{\kappa}{\sqrt{2}} \norm{P_{\ks} W}  \text{ for all } W \in \HS \, . 
        \end{align*}
        In the last step, we have used that $\norm{T} \leq 2$. If $\cB$ is surjective, then we can choose $\ks = \HS = \range (\cB)$ to obtain the lower bound \eqref{eq:Blowbound}.  
    \end{proof}
    \begin{Lemma}\label{lemma:grandnormlowbound}
        Let $\cB$ be surjective, $c_2 \in \R$ with $\abs{c_2} \norm{\cB \cB^*} < 1$ and $c_3 \geq 0$. Consider the potential 
             $\hat{\Pot}(\charX) = {c_3}/{2} \norm{ \cB^* T^{-\frac{1}{2}}(\charX - \Xeq) }^2 - \langle \nabla \Pot(\Xeq), \charX - \Xeq \rangle $
        and assume that $\nabla \Pot$ is Lipschitz. Then there exists a constant $\sigma$ depending on $\cB$ and a constant $K > 0$ depending on $\cB$ and a Lipschitz constant of $\nabla \Pot$ such that for $c_3 > K$, we have 
        \begin{align}
            \norm{\nabla \Pot (Y) + \nabla \hat{\Pot} (Y)} \geq \frac{\sigma^2}{2} (c_3 - K) \norm{ Y - \Xeq } \, , \, Y \in \HS \, . \label{eq:gradnormlowbound}
        \end{align}
    \end{Lemma}
    \begin{proof}
        Since $\cB$ is surjective, Corollary \ref{Cor:Blowbound} implies that there exists $\sigma > 0$ such that \eqref{eq:Blowbound} holds. We define 
            $K \coloneq 2 L /{\sigma^2}$,
        where $L$ is a Lipschitz constant of $\nabla \Pot$. Let us assume that $c_3 > K$. Then we have for $Y \in \HS$
        \begin{align*}
            \langle Y - \Xeq &, \nabla \Pot (Y) + \nabla \hat{\Pot} (Y) \rangle = \langle Y - \Xeq , \nabla \Pot (Y) - \nabla \Pot(\Xeq)  +c_3 T^{-\frac{1}{2}} \cB \cB^* T^{-\frac{1}{2}} (Y - \Xeq) \rangle \\
            &\geq -  L \norm{ Y - \Xeq }^2 + \frac{c_3 \sigma^2}{2} \norm{ Y - \Xeq }^2 
            = \frac{\sigma^2}{2} (c_3 - K) \norm{ Y - \Xeq }^2 \, .
        \end{align*}
         In particular, dividing by $\norm{ Y - \Xeq }$ yields the bound \eqref{eq:gradnormlowbound}.
    \end{proof}
    \begin{Cor}[Convergence properties for different IDA-PBC feedback controllers]\label{cor:IDAPBCconv}
        Let $0 < c_1 $, $c_2 \in \R$ with $\abs{c_2} \norm{\cB \cB^*} < 1$ and $\Xeq \in \HS$. We consider the feedback controller \eqref{eq:genIDAPBCfeedback} and the corresponding closed-loop system \eqref{eq:PHclosedloop} as in Lemma \ref{lemma:IDAPBC} for different functions $F$ and $\hat{\Pot}$. Let $\charX, \, \charV \in C^1([0,\infty) ; \HS )$ be solutions to \eqref{eq:PHclosedloop}.
        \thmpart \label{cor:feedbvelconv}(Convergence for $\hat{\Pot}, \, F$ as in \eqref{eq:FhatV}) We suppose that $\nabla \Pot $ is uniformly continuous and that there exists $\underline{\psi} > 0$ and a closed subspace $\ks \subseteq \range (\cB)$ such that 
            \begin{align}
                \langle \charV, \Psiop(\charX) \charV \rangle \geq \underline{\psi} \norm{ (I - P_{\ks}) \charV }^2 \text{ for all } \charV \in \HS , \label{eq:psikolowbd}
            \end{align}
            where $P_{\ks} $ is the orthogonal projection onto $\ks$.
            Let $c_3 \geq 0$ and assume that $F$ and $\hat{\Pot}$ are given by \eqref{eq:FhatV}. Then $\charV_t \to 0$ and $\nabla \Pot (\charX_t) + \nabla \hat{\Pot}(\charX_t) \to 0 $ in $\HS$ as $t \to \infty$.
            \thmpart \label{cor:feedbvelconvstatad}(Convergence for $\hat{\Pot}, \, F$ as in \eqref{eq:statadFhatV}) 
             Let $\cB$ be surjective, $\nabla \Pot$ be Lipschitz and let $c_3 > 0$. If $F$ and $\hat{\Pot}$ are given by \eqref{eq:statadFhatV}, then $\charV_t \to 0$ and $\nabla \Pot (\charX_t) + \nabla \hat{\Pot}(\charX_t) \to 0 $ in $\HS$ as $t \to \infty $. Moreover, there exists a constant $K > 0$ only depending on $\cB$ and a Lipschitz-constant of $\nabla \Pot$ such that if $c_3 > K$, we also have $\charX_t \to \Xeq$ as $t \to \infty$. 
    \end{Cor}
    \begin{proof}
        We have to verify the assumptions of Theorem \ref{thm:BarbConv} for each feedback controller.
        We start by considering $\hat{\Pot}$ given in \eqref{eq:FhatV}: Then we have 
        \begin{align*}
            \Pot(\charX) + \hat{\Pot}(\charX) = \Pot(\charX) + \frac{c_3}{2} \norm{B^* T^{-\frac{1}{2}} (\charX - \Xeq) }^2 \, .
        \end{align*}
        In particular, since $\Pot$ is bounded from below (see Assumption \ref{ass:control}), we see that $\Pot + \hat{\Pot}$ is also bounded from below. In order to verify Assumption \ref{ass:feedunifcont}, note that 
        \begin{align*}
            \nabla \Pot(\charX) + \nabla \hat{\Pot} ( \charX) =  \nabla \Pot(\charX) + c_3 T^{-\frac{1}{2} } \cB \cB^* T^{-\frac{1}{2}} (\charX - \Xeq) \, .
        \end{align*}
        Since $\nabla \Pot$ is uniformly continuous and $\nabla \hat{\Pot}$ is a bounded linear map, we find that $\nabla \Pot + \nabla \hat{\Pot}$ is also uniformly continuous. Moreover, using the upper bound \eqref{eq:potgrwthbd}, we see that
        \begin{align*}
            \lVert \nabla \Pot(\charX_t) + \nabla \hat{\Pot} ( \charX_t) \rVert 
            &\leq \frac{1}{2} \big(1 + \norm{\nabla \Pot (\charX_t)}^2 \big)  
             + c_3 \norm{T^{-\frac{1}{2} } \cB \cB^* T^{-\frac{1}{2}} (\charX_t - \Xeq)} \\
            &\leq \frac{1}{2} +  \frac{C}{2} (1 + \Pot(\charX_t) -\underline{\mathcal{V}}  ) + c_3 \norm{T^{-\frac{1}{2}} } \norm{\cB} \norm{\cB^* T^{-\frac{1}{2}} (\charX_t - \Xeq)} \\
            &\leq \frac{1}{2} +  \frac{C}{2}(1 + \Pot(\charX_t) -\underline{\mathcal{V}} ) + \frac{c_3}{2} \norm{T^{-\frac{1}{2}} }^2 \norm{\cB}^2 
            + \frac{c_3}{2} \norm{\cB^* T^{-\frac{1}{2}} (\charX_t - \Xeq) }^2\\
            &\leq \frac{1}{2} +  \frac{C}{2} + \frac{c_3}{2} \norm{T^{-\frac{1}{2}} }^2 \norm{\cB}^2 + \max \{C/{2}, 1 \} (\Pot(\charX_t) - \underline{\mathcal{V}} + \hat{\Pot}(\charX_t) ) \\
            &\leq \frac{1}{2} +  \frac{C}{2}  + \frac{c_3}{2} \norm{T^{-\frac{1}{2}} }^2 \norm{\cB}^2 + \max \{C/{2}, 1 \} (H_{\text{cl}}(\charX_0, \charV_0) - \underline{\mathcal{V}})   \, .
        \end{align*}
        It remains to show the positive definiteness Assumption \ref{ass:feedposdef}: Let $\kappa > 0$ be as in Corollary \ref{Cor:Blowbound}. By using \eqref{eq:psikolowbd}, we obtain for $Y, W \in \HS$
        \begin{align}
            \langle W &, (\Psiop(Y) + c_1 T^{- \frac{1}{2}} \cB \cB^* T^{- \frac{1}{2}}) W \rangle
            \geq \underline{\psi } \norm{(I - P_{\ks} )W }^2 + c_1 \norm{\cB^* T^{-\frac{1}{2}} W }^2 \nonumber \\
            &\geq \underline{\psi } \norm{(I - P_{\ks} )W}^2 + c_1 \frac{\kappa^2}{2} \norm{ P_{\ks} W }^2
            \geq \frac{1}{2} \min\{2\underline{\psi}, c_1 \kappa^2 \} \norm{W}^2 \, . \label{eq:verifyposdef}
        \end{align}
        Now, let us consider the Feedback \eqref{eq:statadFhatV} and let $\cB$ be surjective. Let $\sigma > 0$ be as in \eqref{eq:Blowbound}. Thus, we obtain for $Y \in \HS$ and $\varepsilon > 0$
        \begin{align*}
            \Pot(Y) + \hat{\Pot}(Y) &= \Pot(Y) - \langle \nabla \Pot (\Xeq ) , Y - \Xeq \rangle + \frac{c_3}{2} \norm{B^* T^{-\frac{1}{2}} (Y - \Xeq)}^2 \\
            &\geq \underline{\mathcal{V}} - \norm{ \nabla \Pot (\Xeq) } 
            \norm{ Y - \Xeq } + \frac{c_3 \sigma^2 }{4}  \norm{ Y - \Xeq }^2 \\
            &\geq \underline{\mathcal{V}} - \frac{1}{2\varepsilon} \norm{ \nabla \Pot (\Xeq) } + \Big(\frac{c_3 \sigma^2 }{4} - \varepsilon \norm{ \nabla \Pot (\Xeq) }  \Big)  \norm{ Y - \Xeq }^2 \, .
        \end{align*}
        Therefore, we can take $\varepsilon < c_3 \sigma^2 (4 \norm{\nabla \Pot (\Xeq)} )^{-1} $ to see that $\Pot + \hat{\Pot}$ is bounded from below. The uniform continuity and boundedness Assumption \ref{ass:feedunifcont} can be proven similarly as for the previous feedback controller \eqref{eq:FhatV}. Positive definiteness can be seen as in \eqref{eq:verifyposdef} by taking $\ks = \HS$. 
        Now it remains to show that Assumption \ref{ass:feedposconv} is satisfied if $c_3$ is sufficiently large. More precisely, if we choose $K$ as in Lemma \ref{lemma:grandnormlowbound} and if $c_3 > K$, then the lower bound \eqref{eq:gradnormlowbound} implies that Assumption \ref{ass:feedposconv} holds.
    \end{proof}
    By applying the previous corollary, we can now derive convergence results for the continuous time closed-loop systems obtained in Corollary \ref{cor:semexquadvdes}. Recall that in the continuous time setting, we regard $\Delta t$ and $\rho$ as fixed parameters. 
        \begin{Thm}[Convergence of the continuous time systems \eqref{eq:semiexpleulerconttime} and \eqref{eq:semiexpleulerconttimestatad}]\label{thm:semexeulerconttimeconv} 
            We consider the setting from Corollary \ref{cor:semexquadvdes} and let $T \coloneq I - c_2 \cB \cB^*$. 
            \thmpart\label{it:semexctcnst} We suppose that $\nabla \Pot$ is uniformly continuous and that there exists $\underline{\psi} > 0$ and a closed subspace $\ks \subseteq \range (\cB)$ such that 
            \begin{align}
                \langle \charV, \Psiop(\charX) \charV \rangle \geq \underline{\psi} \norm{ (I - P_{\ks}) \charV }^2 \text{ for all } \charV \in \HS \, .
            \end{align}
                Let the parameters $\Delta t > 0, c \geq 0$ and let $\rho$ satisfy 
            \begin{align*}
                0 < \rho < \frac{1}{(\Delta t)^3 (1+ c(\Delta t)^2) \norm{\cB \cB^*}} \, .
            \end{align*}
            Let $\charX, \, \charV \in C^1([0,\infty) ; \HS )$ be solutions to the closed-loop system \eqref{eq:semiexpleulerconttime}. Then $\charV_t \to 0$ and $\nabla \Pot (\charX_t) + c_3 T^{-1} \cB \cB^*(\charX_t - \Xeq) \to 0 $ in $\HS$ as $t \to \infty$.
                \thmpart\label{it:semexctsurj} Let $\cB$ be surjective and $\nabla \Pot$ be Lipschitz. Then there exists a constant $K > 0$ depending on $\cB$ and a Lipschitz constant of $\nabla \Pot$ such that if the parameters $\Delta t, \rho$ and $c$ satisfy 
                \begin{align}
                    0 < \Delta t &< \frac{1}{\sqrt{ K \norm{\cB \cB^*}}}, \ c > \frac{K \norm{\cB \cB^*}}{1 - (\Delta t)^2 K \norm{\cB \cB^*} } , \  
                    \frac{K}{c (\Delta t)^3} &< \rho < \frac{1}{(\Delta t)^3(1+c(\Delta t)^2 ) \norm{\cB \cB^*}} \, , \label{eq:parrestrsemsex}
                \end{align}
                then any solution $\charX, \, \charV \in C^1([0,\infty) ; \HS )$ to the continuous time-system \eqref{eq:semiexpleulerconttimestatad} satisfies $\charV_t \to 0$ and $\charX_t \to \Xeq $ in $\HS$ as $t \to \infty$.
        \end{Thm}
        \begin{proof}
            The convergence result for the system \eqref{eq:semiexpleulerconttime} is a consequence of Corollary \ref{cor:feedbvelconv}. Now let us assume that $\cB$ is surjective. According to Corollary \ref{cor:feedbvelconvstatad}, there exists a constant $K > 0$ depending on a Lipschitz constant of $\nabla \Pot$ and $\cB$ such that if $c_3 > K$ and $c_2 \norm{\cB \cB^*} < 1$, we obtain $\charV_t \stackrel{t \to \infty}{\to} 0$ and $\charX_t \stackrel{t \to \infty}{\to } \Xeq $ in $\HS$. Since $c_2$ and $c_3$ are given by \eqref{eq:c1c3semex}, we obtain the conditions 
            \begin{align*}
                \rho < \frac{1}{(\Delta t)^3(1+c (\Delta t)^2) \norm{\cB \cB^*}} \text{ and } \rho > \frac{K}{c (\Delta t)^3} \, ,
            \end{align*}
            which are satisfied due to \eqref{eq:parrestrsemsex}.
        \end{proof}

        To complete the convergence analysis, we analyze convergence of the discrete time system. The main ingredient is the energy balance of a modified closed-loop Hamiltonian.
        \begin{Thm}[Discrete time energy balance for a modified Hamiltonian using the semi-explicit Euler scheme\label{thm:semexdiscrtimeHepsbal}]
            We consider the setting from Corollary \ref{cor:semexquadvdes} and suppose that $\nabla \Pot$ is Lipschitz. Thus for $\Delta t , \rho > 0$ and $c \geq 0$, the parameters $c_1,c_2,c_3$ are given by \eqref{eq:c1c3semex}. 
             We consider the semi-explicit Euler-discretization \eqref{eq:semi-explEuler0} with step-size $\Delta t > 0$ of the closed-loop systems \eqref{eq:semiexpleulerconttime} and \eqref{eq:semiexpleulerconttimestatad}, that is $\charX^0 = \charX_0$, $\charV^0 = \charV_0$ and
             \begin{subequations}\label{eq:semexeulerits}
             \begin{align}
                 \charX^{l+1} &= \charX^l + \Delta t \charV^{l+1}, \\
                 \charV^{l+1} &= \charV^l \hspace{-1pt} - \Delta t T \Big(\Psiop(\charX^l) \hspace{-1pt} + c_1 T^{-\frac{1}{2}} \cB \cB^* T^{-\frac{1}{2}}  \Big) \charV^l 
                 \hspace{-1pt} - \Delta t T (\nabla \Pot(\charX^l) + \nabla \hat{\Pot}(\charX^l) ) \label{eq:semexeuleritsVup}  .
                 \end{align}
             \end{subequations}
        Moreover, recall that the operator $T$ is defined by $T = I - c_2 \cB \cB^*$ and $H_{\text{cl}}$ is given by \eqref{eq:genHcl}. For $\varepsilon > 0$, we define the modified closed-loop Hamiltonian $H_{\text{cl},\varepsilon}$ by 
            \begin{align}
                H_{\text{cl},\varepsilon} (\charX, \charV) &= H_{\text{cl}}(\charX, \charV) + \varepsilon \langle \charV , T (  \nabla \Pot (\charX) + \nabla \hat{\Pot} (\charX) ) \rangle\label{eq:Hepsdef}  .
            \end{align}
            We consider the following cases: 
            \thmpart\label{ass:Hepsdecr} Suppose there exists $\underline{\psi} > 0$ and a closed subspace $\ks \subseteq \range (\cB)$ such that 
            \begin{align}
                    \langle W, \Psiop(\charX) W \rangle \geq \underline{\psi} \norm{(I- P_{\ks}) W}^2 , \ W, \, \charX \in \HS \, , \label{eq:Psioplowbd}
                \end{align}
            where $P_{\ks}$ is the orthogonal projection onto $\ks$. 
                Let $\hat{\Pot}$ be given by \eqref{eq:semiexnonstatad}. 
                Then there exists a constant $ 0< \alpha < 1$ (which can be chosen as $\alpha = 1/{66}$), such that for every $\gamma \in (0,1] $ and for all $\overline{c} \geq 0$, there exists $h > 0$ (depending on $\overline{c}, \gamma, \underline{\psi}, \norm{\Psiop}_\infty$ a Lipschitz constant of $\nabla \Pot$, $C > 0$ as in \eqref{eq:potgrwthbd} and the control operator $\cB$) such that if
                \begin{align*}
                    0 \leq c \leq \overline{c}  , \,  \frac{\gamma \alpha}{(\Delta t)^3 (1+c(\Delta t)^2) \norm{\cB \cB^*} } &\leq \rho \leq \frac{\alpha }{(\Delta t)^3 (1+c(\Delta t)^2) \norm{\cB \cB^*}} \text{ and } 0 < \Delta t \leq h, 
                \end{align*}
                then $H_{\text{cl}, 3 \Delta t}$ is bounded from below and it holds that 
                \begin{align}
                    H_{\text{cl}, 3 \Delta t} (\charX^{l+1} , \charV^{l+1}) 
                    &\leq  H_{\text{cl}, 3 \Delta t} (\charX^{l} , \charV^{l}) 
                    - \frac{\Delta t}{2} \underline{\psi} \norm{ (I- P_{\ks})\charV^l }^2
                     - \frac{c_1 \Delta t}{4} \norm{\cB^*T^{-\frac{1}{2}} \charV^l}^2 \nonumber \\
                    &\quad- \frac{(\Delta t)^2}{2} \norm{T \big(\nabla \Pot(\charX^l) + \nabla \hat{\Pot}(\charX^l) \big) }^2 \text{ for all } l \in \N_0 \, . \label{eq:semexHepsbound}
                \end{align}
                \thmpart\label{ass:HepsdecrBsurj} Suppose that $\cB$ is surjective and let $\hat{\Pot}$ be given by \eqref{eq:semiexstatad}.
                Then there exists a constant $0 < \alpha < 1$ (which can be chosen as $\alpha = 1/{66}$) and a constant $\nu > 0$ (only depending on $\cB$) such that for every $\gamma \in (0,1]$ and for all $\overline{c} > 0$, there exists $h > 0$ (depending on $\overline{c}, \gamma, \norm{\Psiop}_\infty, \norm{ \nabla \Pot (\Xeq) }$ a Lipschitz constant of $\nabla \Pot$, $C > 0$ as in \eqref{eq:potgrwthbd} and the control operator $\cB$) such that if 
                \begin{align*}
                    0 < c \leq \overline{c}, \, \frac{\gamma \alpha}{(\Delta t)^3 (1+c(\Delta t)^2) \norm{\cB \cB^*} } \leq \rho \leq \frac{\alpha }{(\Delta t)^3 (1+c(\Delta t)^2) \norm{\cB \cB^*}} \text{ and } 0 < \Delta t \leq h  , 
                \end{align*}
                then $H_{\text{cl}, 3 \Delta t}$ is bounded from below and for $l \in \N_0$, we have 
                \begin{align}
                    H_{\text{cl}, 3 \Delta t} (\charX^{l+1} , \charV^{l+1}) 
                    &\leq  H_{\text{cl}, 3 \Delta t} (\charX^{l} , \charV^{l}) 
                     - \nu c_1 \Delta t \norm{\charV^l}^2 
                     - \frac{(\Delta t)^2}{2} \norm{T \big(\nabla \Pot(\charX^l) + \nabla \hat{\Pot}(\charX^l) \big) }^2 \label{eq:semexHepsboundsurj} \, .
                \end{align}
        \end{Thm}
        \begin{proof}[Proof outline]
            The proof of Theorem \ref{thm:semexdiscrtimeHepsbal} is rather technical, so we only summarize its main steps here. The full proof is given in the Appendix \ref{sec:appendix}.
            We start by computing the discrete time energy balance of the modified Hamiltonian, that is, computing 
            $H_{\text{cl}, \varepsilon}(\charX^{l+1}, \charV^{l+1}) - H_{\text{cl}, \varepsilon}(\charX^l, \charV^l)$. 
            The result consists of eight terms, each of which is bounded above separately. We conclude by choosing the parameters $\rho, \Delta t, c$ and $\varepsilon = 3 \Delta t$ such that \eqref{eq:semexHepsbound} respectively \eqref{eq:semexHepsboundsurj} holds and $H_{\text{cl}, 3 \Delta t}$ is bounded from below. 
        \end{proof}
        \begin{Rem}
            If $\nabla \Pot + \nabla \hat{\Pot}$ is Lipschitz, the modified Hamiltonian can also be used as a Lyapunov function for the continuous time system \eqref{eq:PHclosedloop}. To do so, one either has to require that $\Pot + \hat{\Pot}$ is $C^2$, or, to avoid this assumption, define $H_{\text{cl}, \varepsilon}$ as 
            \begin{align*}
                H_{\text{cl}, \varepsilon}(\charX, \charV) = \frac{1}{2} \langle \charV, T^{-1} \charV \rangle + \Pot(\charX + \varepsilon \charV) + \hat{\Pot}(\charX + \varepsilon \charV) \, .
            \end{align*}
            But since it is not apparent how to formulate the closed-loop system \eqref{eq:PHclosedloop} as a port-Hamiltonian system using $\nabla H_{\text{cl}, \varepsilon}$, we have avoided this approach on the continuous time level. Moreover, we have seen in Corollary \ref{cor:feedbvelconv} that it is sufficient to require uniform continuity of $\nabla \Pot$ instead of Lipschitz continuity. 
        \end{Rem}
        \begin{Cor}[Convergence of the discrete time system \eqref{eq:semexeulerits}]\label{thm:semexeulerdiscrtimeconv}
            We consider the setting given in Theorem \ref{thm:semexdiscrtimeHepsbal}.
            \thmpart\label{it:pslowb} We suppose that \eqref{eq:Psioplowbd} holds and let $\hat{\Pot}$ be given by \eqref{eq:semiexnonstatad}. Let \par \noindent
            $(\charX^l)_{l \in \N_0 }, \, (\charV^l)_{l \in \N_0} \subseteq \HS $ be iterates of the discrete time closed-loop system \eqref{eq:semexeulerits}.
                Then there exists a constant $0 < \alpha < 1$ such that for every $\gamma \in (0,1]$ and all $\overline{c} \geq 0$, there exists $0 < h$ such that if $\Delta t, c $ and $\rho$ satisfy 
            \begin{align}
                0 < \Delta t \leq h, \, \frac{\gamma \alpha}{(\Delta t)^3 (1+ c(\Delta t)^2) \norm{\cB \cB^*}}  \leq \rho \leq \frac{\alpha}{(\Delta t)^3 (1+ c(\Delta t)^2) \norm{\cB \cB^*}} \text{ and } 0 \leq c \leq \overline{c} , \label{eq:parrestr}
            \end{align}
             then $\charV^l \to 0$ and $\nabla \Pot (\charX^l) + c \rho (\Delta t)^3 T^{-1} \cB \cB^*(\charX^l - \Xeq) \to 0 $ in $\HS$ as $l \to \infty$.
                \thmpart\label{it:Bsurj} Let $\cB$ be surjective and $\hat{\Pot}$ be given by \eqref{eq:semiexstatad}. Let $(\charX^l)_{l \in \N_0 }, \, (\charV^l)_{l \in \N_0} \subseteq \HS $ be iterates of the closed-loop system \eqref{eq:semexeulerits}.
                Then there exist constants $0 < \alpha < 1$ and $K > 0$ such that for all $\gamma \in (0,1]$ and all $\overline{c}, \underline{c}$ satisfying
                \begin{align}
                    \overline{c} \geq \underline{c} > \frac{K \norm{\cB \cB^*}}{\gamma \alpha} > 0 \, , \label{eq:barccond}
                \end{align}
                there exists $0 < h$ such that if $\Delta t, \rho$ and $c$ satisfy 
                \begin{align}
                    0 < \Delta t \leq h, \,
                       \frac{\gamma \alpha}{(\Delta t)^3 (1+ c(\Delta t)^2) \norm{\cB \cB^*}}  &\leq \rho 
                       \leq \frac{\alpha}{(\Delta t)^3 (1+ c(\Delta t)^2) \norm{\cB \cB^*}}
                      \text{ and } 
                      \underline{c} \leq c \leq \overline{c} , \label{eq:parrestrBsurj}
                \end{align}
                then the iterates $(\charX^l)_{l \in \N_0 }, (\charV^l)_{l \in \N_0} \subseteq \HS$ of the discrete time closed-loop system \eqref{eq:semexeulerits} satisfy $\charV^l \to  0$ and $\charX^l \to \Xeq $ in $\HS$ as $l \to \infty$.
        \end{Cor}
        \begin{proof}
            We start by proving Assertion \ref{it:pslowb}. Fix $\gamma \in (0,1]$ and $\overline{c} \geq 0$. According to Theorem \ref{thm:semexdiscrtimeHepsbal}, there exists $0 < h$ such that if the parameters $c, \rho, \Delta t$ satisfy \eqref{eq:parrestr}, then \eqref{eq:semexHepsbound} holds and $H_{\text{cl}, 3 \Delta t}$ is bounded from below. Hence, for $P \in \N_0$, we have 
            \begin{align*}
                  &H_{\text{cl}, 3 \Delta t} (\charX^0, \charV^0) - H_{\text{cl}, 3 \Delta t} (\charX^P, \charV^P)
                = \sum_{l=0}^{P-1} \big[ H_{\text{cl}, 3 \Delta t} (\charX^l, \charV^l) - H_{\text{cl}, 3 \Delta t} (\charX^{l+1}, \charV^{l+1} ) \big] \\
                &\geq  \Delta t \sum_{l=0}^{P-1} \Big[ 
                \frac{1}{2} \underline{\psi} \norm{(I-P_{\ks} )\charV^l}^2
                     + \frac{c_1}{4} \norm{\cB^*T^{-\frac{1}{2}} \charV^l}^2  
                     +
                     \frac{\Delta t}{2} \norm{T (\nabla \Pot(\charX^l) + \nabla \hat{\Pot}(\charX^l) )}^2 \Big] \, .
            \end{align*}
            Since $H_{\text{cl}, 3 \Delta t}$ is bounded from below, it follows that 
            \begin{align*}
                \sum_{l=0}^{\infty} \Big[ 
                \frac{1}{2} \underline{\psi} \norm{(I-P_{\ks} )\charV^l}^2
                     + \frac{c_1}{4} \norm{\cB^*T^{-\frac{1}{2}} \charV^l}^2   +
                     \frac{\Delta t}{2} \norm{T (\nabla \Pot(\charX^l) + \nabla \hat{\Pot}(\charX^l) ) }^2 \Big] 
                     < \infty .
            \end{align*}
            In particular, we obtain that 
            \begin{align*}
                (I-P_{\ks})\charV^l , \ \cB^*T^{-\frac{1}{2}} \charV^l , \ \nabla \Pot(\charX^l) + \nabla \hat{\Pot}(\charX^l) \to 0 \text{ in } \HS \text{ as } l \to \infty \, .
            \end{align*}
            According to Corollary \ref{Cor:Blowbound}, convergence of $\cB^*T^{-\frac{1}{2}} \charV^l$ implies that $P_{\ks} \charV^l \to 0$ and therefore we obtain $\charV^l \to 0$ as $l \to \infty$.
            Now, let us prove Assertion \ref{it:Bsurj}. Fix $\gamma \in (0,1]$ and $\overline{c} \geq 0$. By Lemma \ref{lemma:grandnormlowbound}, there exists $\sigma, \, K > 0$ such that \eqref{eq:gradnormlowbound} holds for $c_3 > K$. Theorem \ref{thm:semexdiscrtimeHepsbal} implies that, there exists $0 < h$ such that if $c, \rho, \Delta t$ satisfy \eqref{eq:parrestrBsurj}, then \eqref{eq:semexHepsboundsurj} holds and $H_{\text{cl}, 3 \Delta t}$ is bounded from below. Moreover, we have 
            \begin{align*}
                c_3 = c \rho (\Delta t)^3 \geq c \frac{\gamma \alpha }{\norm{\cB \cB^*} (1 + c (\Delta t)^2) } 
                \geq \underline{c} \frac{\gamma \alpha }{\norm{\cB \cB^*} (1 + \overline{c} h^2) } \, .
            \end{align*}
            By possibly taking a smaller value of $h$, we can therefore assume that for $0 < \Delta t \leq h$, we have $c_3 > K$. 
            By using a similar argument as in the first assertion, we obtain that 
            \begin{align*}
                \sum_{l=0}^{\infty} \Big[ \nu c_1 \norm{\charV^l}^2
                    + \frac{\Delta t}{2} \norm{T (\nabla \Pot(\charX^l) + \nabla \hat{\Pot}(\charX^l) ) }^2 \Big] < \infty \, .
            \end{align*}
            Since $T$ has a bounded inverse and due to the bound \eqref{eq:gradnormlowbound}, it follows that 
            \begin{align*}
                 \sum_{l=0}^{\infty} \Big[ \nu c_1 \norm{\charV^l}^2
                    + \frac{\Delta t}{8\norm{T}^{-1}} \sigma^4 (c_3 - K)^2 \norm{ \charX^l - \Xeq }^2 \Big] < \infty 
            \end{align*}
            and therefore $\charV^l \to 0$ and $\charX^l \to \Xeq$ in $\HS$ as $l \to \infty$. 
        \end{proof}

\section{Application to the Cucker--Smale model}\label{sec:aplications}

    In this final section, we apply our results to the Cucker--Smale model. 
            To do so, we impose conditions on the alignment function $\psi$ and the potential $\mathcal{V}$ which ensure that Assumption \ref{ass:control} and the additional assumptions in the previous section are satisfied.
            \begin{Notation}
        Let $\mu_0$ be a Borel probability measure on $\R^d \times \R^d$. For $\charV \in \HS \coloneq L^2(\msolinit; \R^d)$, we define the mean-velocity $\overline{\charV} \in \R^d$ via 
        \begin{align*}
            \overline{\charV} \coloneq \int \charV(x,v) \, \mathrm{d}\msolinit(x,v) \, .
        \end{align*}
        We denote the space of constant functions $\{v \mathds{1} \mid v \in \R^d \}$ by $\cnst$. Then the orthogonal projection in $L^2(\msolinit; \R^d)$ onto $\cnst$ reads $P_{\cnst} \charV = \overline{\charV} \mathds{1}$.
    \end{Notation}
    \begin{Prop}\label{prop:PsiandPotprop}
        Let $\psi \in C(\R_+; \R_+)$ be bounded and $\mathcal{V} \in C^1(\R^d; \R)$ with $\nabla \mathcal{V}$ antisymmetric. Let $\B \in \R^{d \times d}$ be symmetric and positive semidefinite. Suppose that $\nabla \mathcal{V}$ has at most linear growth, that is, there exists $K \geq 0$ such that 
        \begin{align}
            \abs{\nabla \mathcal{V} (x) } \leq K (1+ \abs{x}) \text{ for all } x \in \R^d \, . \label{eq:lingrwthV}
        \end{align}
        Let $\msolinit$ be a compactly supported Borel probability measure on $\R^d \times \R^d$ and let $\Psiop$ and $\Pot$ be defined by \eqref{eq:defCSPsiop} and \eqref{eq:PotextAg}. Let $\wps $ denote either $L^\infty(\msolinit; \R^d) $ or $ C(\supp(\msolinit) ; \R^d) $ endowed with the (essential) supremum norm. Then the following statements hold: 
        \thmpart\label{it:psisa} For every $\charX \in L^2(\msolinit; \R^d)$, $\Psiop(\charX) \in \mathscr{L}(L^2(\msolinit; \R^d))$ is self-adjoint, positive semidefinite and satisfies $\norm{\Psiop(\charX)} \leq \norm{\psi}_\infty$. Moreover,  $\cnst \subseteq \ker(\Psiop(\charX))$ and it holds that 
            \begin{align}
                \underline{\psi} \norm{\charV - \overline{\charV} \mathds{1}}^2 \leq \langle \charV, \Psiop(\charX) \charV \rangle \leq \norm{\psi}_\infty \norm{\charV - \overline{\charV} \mathds{1}}^2 \, \text{for} \, \charV \in L^2(\msolinit; \R^d), \,  \underline{\psi} \coloneq \inf_{x \in \R_+} \psi(x) \, . \nonumber
            \end{align}
            In particular, if $\underline{\psi} > 0$, then \eqref{eq:Psioplowbd} holds for $\ks = \cnst$.
            \thmpart\label{it:psirestr} For $X \in \wps$, $\Psiop(\charX)$ restricts to a bounded linear operator on $\wps$ and the map $\Psiop(\cdot) \colon \wps \to \mathscr{L}(\wps)$ is bounded. If $\psi$ is locally Lipschitz, then the map $\charX \mapsto \Psiop(\charX)$ is also locally Lipschitz from $\wps$ to $\mathscr{L}(\wps )$.
            \thmpart\label{it:VC1} The potential $\Pot$ is well-defined, $\Pot \in C^1(L^2(\msolinit;\R^d);\R)$ and $\nabla \Pot$ is given by 
            \begin{align}
                \nabla \Pot(\charX)(x,v) = \int \nabla \mathcal{V}(\charX(x,v) - \charX(x^\prime, v^\prime) ) \, \mathrm{d}\msolinit(x^\prime,v^\prime) 
                + \B \charX(x,v)  \, . \label{eq:nablaPot}
            \end{align} 
            \thmpart\label{it:Vgrwth} If $\mathcal{V}$ is bounded from below by a constant $\underline{\mathcal{V}}$ and there exists $C \geq 0$ such that 
            \begin{align}
                \abs{\nabla \mathcal{V}(x)}^2 \leq C (1 + \mathcal{V}(x) - \underline{\mathcal{V}} )  \text{ for all } x \in \R^d , \label{eq:calVgrwth}
            \end{align}
            then \eqref{eq:potgrwthbd} is satisfied (with constant $C$ replaced by $ 4 \max \{ C, \norm{B} \} $).
            \thmpart\label{it:VlocLip} For $\charX \in \wps$, we have $\nabla \Pot(\charX) \in \wps$ and $\nabla \Pot$ satisfies the $\wps$-linear growth bound 
            \begin{align}
                \norm{\nabla \Pot( \charX) }_{\wps} \leq \max \{ 2 K, \norm{B} \} (1 + \norm{\charX}_{\wps} ) \label{eq:wpslingrwth} \, .
            \end{align}
            Moreover, if $\nabla \mathcal{V}$ is locally Lipschitz, then $\nabla \Pot \colon \wps \to \wps$ is locally Lipschitz.
            \thmpart\label{it:VLip} If $\nabla \mathcal{V}$ is Lipschitz with Lipschitz constant $L$, then \par \noindent
            $\nabla \Pot  \in C(L^2(\msolinit; \R^d) ; L^2(\msolinit; \R^d) )$ is Lipschitz with Lipschitz constant $2L + \norm{\B}$.
            \thmpart\label{it:Vunifcont} If $\nabla \mathcal{V} \in C(\R^d; \R^d)$ is uniformly continuous and bounded, then \par \noindent
            $\nabla \Pot  \in C(L^2(\msolinit; \R^d) ; L^2(\msolinit; \R^d) )$ is uniformly continuous.
            \thmpart\label{it:Bextag} If the control operator $\cB$ is given by \eqref{eq:cbextAg}, $\wps$ is an invariant subspace of $\cB \cB^*$. 
    \end{Prop}
     Proposition \ref{prop:PsiandPotprop} follows by direct computation, so we omit its proof. The following theorem summarizes the conditions under which the abstract results apply to the Cucker--Smale model.
    \begin{Thm}[Application to the Cucker--Smale model] \label{Thm:appMFCS}
        Consider the Cucker--Smale model introduced in Examples \ref{ex:CSmf} and \ref{ex:CsextAg}. Let $\Xeq \in L^2(\msolinit; \R^d)$, $\psi \in C(\R_+; \R_+)$ be bounded and let $\mathcal{V} \in C^1(\R^d; \R)$ satisfy \eqref{eq:lingrwthV}, \eqref{eq:calVgrwth} with $\nabla \mathcal{V}$ antisymmetric. 
        \thmpart Let $\wps$ be either $L^\infty(\msolinit; \R^d)$ or $C(\supp(\mu_0);\R^d) $. Assume that $\Xeq, \charX_0, \charV_0 \in \wps$ and that $\psi$ and $\nabla \mathcal{V}$ are locally Lipschitz. In the case of Example \ref{ex:CSmf}, we additionally assume that $\wps$ is invariant under $\cB \cB^*$.
        Then for every $\rho, \Delta t, c \geq 0$, there exists a unique solution $\charX, \charV \in C^1([0,\infty); \wps)$ to each of the continuous time systems \eqref{eq:semiexpleulerconttime} and \eqref{eq:semiexpleulerconttimestatad}.
            \thmpart 
            Suppose that $\psi(x) \geq \underline{\psi} > 0, \, x \in \R_+$ and let $\nabla \mathcal{V}$ be either Lipschitz or bounded and uniformly continuous. In the case of Example \ref{ex:CSmf}, assume that $\cnst \subseteq \range(\cB)$ and that $\range (\cB)$ is closed. In the case of Example \ref{ex:CsextAg}, let $\B$ be surjective.
            Then the conclusions
            of Theorem \ref{it:semexctcnst} hold. Moreover, if $\nabla \mathcal{V}$ is Lipschitz, then the conclusions of Corollary \ref{it:pslowb} hold.
            \thmpart\label{thm:applicationCSsurjective} In the case of Example \ref{ex:CSmf}, suppose that $\cB$ is surjective and that $\nabla \mathcal{V}$ is Lipschitz. Then the conclusions of Theorem \ref{it:semexctsurj} and Corollary \ref{it:Bsurj} hold. 
    \end{Thm}
    \begin{Rem}
        We note that there is no analogue of Theorem \ref{thm:applicationCSsurjective} for Example \ref{ex:CsextAg}. Indeed, if $\cB$ is given by \eqref{eq:cbextAg}, then $\cB$ is surjective only if $\B$ is surjective and $\msolinit$ is a Dirac measure.
    \end{Rem}

    \appendix
\section{Proof of Theorem \ref{thm:semexdiscrtimeHepsbal}}\label{sec:appendix}

\begin{proof}[Proof of Theorem \ref{thm:semexdiscrtimeHepsbal}]
     We proceed by following the steps outlined in Section \ref{sec:WpandAs}. \par \noindent
     \textbf{Computation of the energy balance:} Fix $\varepsilon > 0$. Moreover, we will always assume that $0 \leq \Delta t, \rho, c$ and that $c_2$ satisfies $c_2 \norm{\cB \cB^*} < 1$. The latter assumption will be justified later. Consequently, we have $\norm{T} \leq 2$, $\norm{T^{1/{2} }} \leq \sqrt{2}$ and $T$ has a bounded inverse.
            Moreover, we can always assume that \eqref{eq:Psioplowbd} holds. Indeed, in the second part of the proof where we assume $\cB$ to be surjective, we can take $\ks = \HS$ and arbitrary $\underline{\psi} > 0$.
            For $l \in \N_0$, we introduce the shorthand notation 
            \begin{align}
                D^l \coloneq \Psi(\charX^l) + c_1 T^{- \frac{1}{2}} \cB \cB^* T^{- \frac{1}{2}} \text{ and } F^l \coloneq - (\nabla \Pot(\charX^l) + \nabla \hat{\Pot}(\charX^l) ) \, . \label{eq:defDlFl}
            \end{align}
            Then the velocity update \eqref{eq:semexeuleritsVup} reads 
                $\charV^{l+1} = \charV^l - \Delta t T (D^l \charV^l - F^l)$ .
            Since $T^{-1}$ is self-adjoint, we obtain 
            \begin{align}
                \frac{1}{2} \langle \charV^{l+1} , T^{-1} \charV^{l+1} \rangle &= \frac{1}{2} \langle \charV^{l} , T^{-1} \charV^{l} \rangle 
                - \Delta t \langle \charV^{l}, D^l \charV^l \rangle 
                + \Delta t \langle \charV^l, F^l \rangle 
                + \frac{1}{2} (\Delta t)^2 \norm{T^{\frac{1}{2}} (D^l \charV^l - F^l)  }^2 \, , \label{eq:kinupd} \\ 
                \Pot(\charX^{l+1}) + \hat{\Pot}(\charX^{l+1}) &= \Pot(\charX^{l}) + \hat{\Pot}(\charX^{l}) - \Delta t \langle F^l, \charV^{l} \rangle + (\Delta t)^2 \langle F^l, T(D^l \charV^l - F^l) \rangle \nonumber  \\
                &\quad + \big( \Pot(\charX^{l+1}) + \hat{\Pot}(\charX^{l+1}) - \Pot(\charX^{l}) - \hat{\Pot}(\charX^{l}) + \Delta t \langle F^l, \charV^{l+1} \rangle \big) \, , \label{eq:potup} \\
                - \varepsilon \langle \charV^{l+1}, T F^{l+1} \rangle &= - \varepsilon \langle \charV^l - \Delta t T (D^l \charV^l - F^l) , T (F^l + (F^{l+1} - F^l) ) \rangle \nonumber  \\ 
                &= - \varepsilon \langle \charV^l, T F^l \rangle - \varepsilon \langle \charV^l, T (F^{l+1} - F^l) \rangle 
                + \varepsilon \Delta t \langle TD^l \charV^l , T F^l \rangle \nonumber \\
                &\quad - \varepsilon \Delta t \norm{T F^l}^2 
                + \varepsilon \Delta t \langle T(D^l \charV^l - F^l) , T (F^{l+1} - F^l) \rangle \, . \label{eq:VFupd}
            \end{align}
            Taking the sum of \eqref{eq:kinupd}, \eqref{eq:potup} and \eqref{eq:VFupd} yields the energy balance of $H_{\text{cl}, \varepsilon}$
            \begin{align}
                H_{\text{cl}, \varepsilon}(\charX^{l+1}, \charV^{l+1})   
                &= H_{\text{cl}, \varepsilon}(\charX^{l}, \charV^{l}) - \Delta t \langle \charV^l , D^l \charV^l \rangle 
                - \varepsilon \Delta t \norm{TF^l}^2 
                + \varepsilon \Delta t \langle T D^l V^l , T F^l \rangle \nonumber \\
                &\quad - \varepsilon \langle \charV^l, T(F^{l+1} - F^l) \rangle 
                + \frac{1}{2} (\Delta t)^2 \norm{T^{\frac{1}{2}} (D^l \charV^l - F^l)  }^2 
                + (\Delta t)^2 \langle F^l, T(D^l \charV^l - F^l) \rangle \nonumber \\
                &\quad + \big( \Pot(\charX^{l+1}) + \hat{\Pot}(\charX^{l+1}) - \Pot(\charX^{l}) - \hat{\Pot}(\charX^{l}) + \Delta t \langle F^l, \charV^{l+1} \rangle \big) \nonumber \\
                &\quad + \varepsilon \Delta t \langle T(D^l \charV^l - F^l) , T (F^{l+1} - F^l) \rangle \nonumber  \\
                &\eqqcolon H_{\text{cl}, \varepsilon}(\charX^{l}, \charV^{l}) + \RN{1} + \RN{2} + \ldots + \RN{8} \, . \label{eq:semexHepsenbalance}
            \end{align}
            We will now derive upper bounds for each of the above terms $\RN{1}, \ldots, \RN{8}$.\\
            \textbf{Derivation of upper bounds:}  Recall that according to Corollary \ref{Cor:Blowbound}, there exists a constant $\kappa > 0$ only depending on $\cB$ such that 
            \begin{align}
                \norm{\cB^* T^{-\frac{1}{2}} \charV^l } \geq \frac{\kappa}{\norm{T^{\frac{1}{2}}}} \norm{P_{\ks} \charV^l } \geq \frac{\kappa}{\sqrt{2}} \norm{P_{\ks} \charV^l } \, . \label{eq:hBmeanlowbound}
            \end{align}
            By Assumption \ref{ass:control}, we have $\norm{\Psi(Y)}_{\mathscr{L}(\HS) } \leq \norm{\Psiop}_\infty$ for $Y \in \HS$. Thus, for $l \in \N_0$,
            \begin{align}
                \norm{T^{\frac{1}{2}} D^l \charV^l } &\leq \sqrt{2} \norm{\Psiop}_\infty \big( \norm{ (I - P_{\ks} ) \charV^l } + \norm{P_{\ks} \charV^l } \big) + c_1 \norm{\cB} \norm{\cB^*T^{-\frac{1}{2}} \charV^l}  \nonumber \\ 
                &\leq \sqrt{2} \norm{\Psiop}_\infty \norm{ (I - P_{\ks} ) \charV^l } + \Big(c_1 \norm{\cB} + \frac{2}{\kappa} \norm{\Psiop}_\infty \Big) \norm{\cB^*T^{-\frac{1}{2}} \charV^l} \, ,
                \\
                  \norm{T D^l \charV^l } &\leq 2 \norm{\Psiop}_\infty \norm{(I - P_{\ks} )\charV^l } + \sqrt{2} \Big(c_1 \norm{\cB} + \frac{2}{\kappa} \norm{\Psiop}_\infty \Big) \norm{\cB^*T^{-\frac{1}{2}} \charV^l} \, . \label{eq:TDVbound}
            \end{align}
            The first line follows from the definition of $D^l$ in \eqref{eq:defDlFl}, while in the second line, we used \eqref{eq:hBmeanlowbound}. Moreover, we note that $\nabla \Pot + \nabla \hat{\Pot}$ is Lipschitz with a Lipschitz constant
            \begin{align}
                L = \begin{cases}
                     L_{\nabla \Pot} + c_3 \norm{T^{-1}} \norm{\cB \cB^*}, \, &\hat{\Pot} \text{ is given by } \eqref{eq:semiexnonstatad}, \\ 
                    L_{\nabla \Pot}+ c_3 \norm{T^{-1}} \norm{\cB \cB^*} +  \norm{ \nabla \Pot (\Xeq) }, \, &\hat{\Pot} \text{ is given by } \eqref{eq:semiexstatad} \, ,
                \end{cases} \label{eq:Lipc}
            \end{align}
            where $L_{\nabla \Pot}$ is a Lipschitz constant of $\nabla \Pot$.
            Consequently, we deduce the following norm bounds for the updated velocity $\charV^{l+1}$ and the force difference $F^{l+1} - F^l$:
            \begin{align}
                \big\lVert \charV^{l+1} \big\rVert 
                &= \norm{ \charV^l -\Delta t T D^l\charV^l + \Delta t T F^l } \nonumber \\
                &\leq (1 + 2\Delta t \norm{\Psiop}_\infty ) \norm{(I-P_{\ks} )\charV^l} + \norm{ P_{\ks} \charV^l }
                + \Delta t \norm{T F^l} \nonumber  \\
                &\quad + \sqrt{2} \Delta t  \Big(c_1 \norm{\cB} + \frac{2}{\kappa} \norm{\Psiop}_\infty \Big) \norm{\cB^*T^{-\frac{1}{2}} \charV^l}  
                  \nonumber \\
                &\leq (1 + 2\Delta t \norm{\Psiop}_\infty ) \norm{(I-P_{\ks} )\charV^l} \nonumber \\ 
                &\quad 
                + \sqrt{2}\Big(
                \frac{1}{\kappa} +  \Delta t \Big(c_1 \norm{\cB} + \frac{2}{\kappa} \norm{\Psiop}_\infty \Big) \Big) \norm{\cB^*T^{-\frac{1}{2}} \charV^l} 
                + \Delta t \norm{TF^l} \nonumber \\
                &\eqqcolon K_{\Psiop} \norm{(I-P_{\ks} )\charV^l} + K_{\cB} \norm{\cB^*T^{-\frac{1}{2}} \charV^l} 
                +  \Delta t \norm{TF^l}
                 \, , \label{eq:semexVlpbound} \\
                \big\lVert F^{l+1} - F^l \big\rVert &= \norm{\nabla \Pot(\charX^l) + \nabla \hat{\Pot}(\charX^l) -  ( \nabla \Pot(\charX^{l+1}) + \nabla \hat{\Pot}(\charX^{l+1} ) ) } 
                \leq  \Delta t L \norm{\charV^{l+1}} \nonumber \\
                &\leq \Delta t L \Big( K_{\Psiop} \norm{(I-P_{\ks} )\charV^l} + K_{\cB} \norm{\cB^*T^{-\frac{1}{2}} \charV^l} + \Delta t \norm{TF^l} \Big) \, . \label{eq:semexFldiff}
            \end{align}
            The second line follows from \eqref{eq:TDVbound}, while in \eqref{eq:semexFldiff}, we have used that $\nabla \Pot + \nabla \hat{\Pot} $ is Lipschitz with Lipschitz constant $L$. The last step follows from the velocity bound \eqref{eq:semexVlpbound}.
            We are now ready to find bounds for the eight terms in the energy balance \eqref{eq:semexHepsenbalance}. 
            The terms of order $\Delta t$ in the energy balance satisfy the upper bounds
            \begin{align}
                \RN{1} &= - \Delta t \langle \charV^l , D^l \charV^l \rangle \leq - \Delta t \big(\underline{\psi} \norm{(I-P_{\ks} )\charV^l}^2 + c_1 \norm{\cB^*T^{-\frac{1}{2}} \charV^l}^2 \big) \, ,\label{eq:termest1} \\
                \RN{3} &= \varepsilon \Delta t \langle T D^l \charV^l , T F^l \rangle 
                \leq  \varepsilon \Delta t \norm{T D^l\charV^l} \norm{T F^l} \nonumber \\
                &\leq \varepsilon \Delta t \Big( 2 \norm{\Psiop}_\infty \norm{(I-P_{\ks} )\charV^l} + \sqrt{2} \Big(c_1 \norm{\cB} + \frac{2}{\kappa} \norm{\Psiop}_\infty \Big) \norm{\cB^*T^{-\frac{1}{2}} \charV^l} \Big) \norm{TF^l} \nonumber \\
                &\leq \frac{\varepsilon \Delta t \tau }{2} \norm{T F^l}^2 + \frac{2\varepsilon \Delta t}{\tau}  \Big( 2 \norm{\Psiop}_\infty^2 \norm{(I-P_{\ks} )\charV^l}^2 
                +  2 \Big( c_1^2 \norm{\cB}^2 + \frac{4}{\kappa^2} \norm{\Psiop}_\infty^2 \Big) \norm{\cB^*T^{-\frac{1}{2}} \charV^l}^2 \Big)\, ,   \label{eq:termest3} 
            \end{align}
            for all  $\tau > 0$. The estimate in the first line is due to \eqref{eq:Psioplowbd}, while to obtain the estimate \eqref{eq:termest3}, we have used \eqref{eq:TDVbound} and, in the last step, Young's inequality in combination with the Cauchy Schwarz inequality $(a+b)^2 \leq 2 (a^2 + b^2)$ for $a,b \in \R$.
            Subsequently, we use \eqref{eq:hBmeanlowbound} and \eqref{eq:semexFldiff}  to obtain    
            \begin{align}
                \RN{4} &= - \varepsilon \langle \charV^l, T(F^{l+1} - F^l) \rangle \leq 2 \varepsilon \norm{\charV^l} \norm{F^{l+1} - F^l} 
                \leq 2 \varepsilon  \Big(\norm{(I-P_{\ks} )\charV^l} + \norm{P_{\ks} \charV^l} \Big)
                \norm{F^{l+1} - F^l} \nonumber \\
                &\leq 2 \varepsilon \Delta t L \Big(\norm{(I-P_{\ks} )\charV^l} + \frac{\sqrt{2}}{\kappa} \norm{\cB^*T^{-\frac{1}{2}} \charV^l} \Big)  \Big( K_{\Psiop} \norm{(I-P_{\ks} )\charV^l} + K_{\cB} \norm{\cB^*T^{-\frac{1}{2}} \charV^l} + \Delta t \norm{TF^l} \Big)
                \nonumber \\
                &=  2\varepsilon \Delta t L \Big(  K_{\Psiop} \norm{(I-P_{\ks} )\charV^l}^2 
                +  \Big( K_{\cB} \norm{\cB^*T^{-\frac{1}{2}} \charV^l} + \Delta t \norm{TF^l} \Big) \norm{(I-P_{\ks} )\charV^l} \nonumber \\
                &\quad 
                + \frac{\sqrt{2}}{\kappa}  K_{\cB} \norm{\cB^*T^{-\frac{1}{2}} \charV^l}^2 
                + \frac{\sqrt{2}}{\kappa}  \Big(
                K_{\Psiop} \norm{(I-P_{\ks} )\charV^l} 
                + \Delta t \norm{TF^l}
                \Big) \norm{\cB^*T^{-\frac{1}{2}} \charV^l} \Big) \nonumber \\
                &\leq 2\varepsilon \Delta t L \Big( 
                     \Big[ K_{\Psiop} + \frac{K_{\cB}}{2}
                + \frac{\sqrt{2}}{2 \kappa }  K_{\Psiop}
                +\frac{\Delta t}{2}
                \Big] \norm{(I-P_{\ks} )\charV^l}^2 \nonumber \\
                     &\quad+ 
                     \Big[ \frac{\sqrt{2}}{\kappa} K_{\cB} 
                + \frac{K_{\cB}}{2} 
                + \frac{\sqrt{2}}{2 \kappa }  K_{\Psiop}
                + \frac{\sqrt{2}\Delta t}{2\kappa } 
                \Big] \norm{\cB^*T^{-\frac{1}{2}} \charV^l}^2 
                + \frac{\sqrt{2}}{\kappa} \Delta t \norm{TF^l}^2 \Big)
                \, . \nonumber
            \end{align}
            In the last step, we have applied Young's inequality $ab \leq (a^2 + b^2) /{2}$ for $a,b \geq 0$ several times.
            For the second order $\Delta t$ terms in \eqref{eq:semexHepsenbalance}, we obtain
            \begin{align*}
                \RN{5} &=  \frac{1}{2} (\Delta t)^2 \norm{T^{\frac{1}{2}} (D^l \charV^l - F^l)  }^2  
                 \leq \frac{1}{2} (\Delta t)^2 \Big( \norm{T^{\frac{1}{2}}D^l \charV^l} + \Big(\langle F^l, T F^l \rangle\Big)^{\frac{1}{2}} \Big)^2 \\
                 &\leq  (\Delta t)^2 \Big( \Big( \sqrt{2} \norm{\Psiop}_\infty \norm{(I-P_{\ks} )\charV^l} + \Big(c_1 \norm{\cB} + \frac{2}{\kappa} \norm{\Psiop}_\infty \Big) \norm{\cB^*T^{-\frac{1}{2}} \charV^l}\Big)^2 
                 + \norm{T^{-1}} \norm{T F^l}^2\Big)  \\
                &\leq (\Delta t)^2 \Big( 
                    4  \norm{\Psiop}_\infty^2 \norm{(I-P_{\ks} )\charV^l}^2
                    + 4 \Big( c_1^2 \norm{\cB}^2 + \frac{4}{ \kappa^2 } \norm{\Psiop}_\infty^2 \Big) \norm{\cB^*T^{-\frac{1}{2}} \charV^l}^2 
                    + \norm{T^{-1}} \norm{TF^l}^2  
                \Big) \, , \\
                 \RN{6} &= (\Delta t)^2 \langle F^l, T(D^l \charV^l - F^l) \rangle 
                 \leq (\Delta t)^2 \langle T F^l, D^l \charV^l \rangle 
                 \leq (\Delta t)^2   \norm{D^l \charV^l} \norm{TF^l} \nonumber \\
                 &\leq  (\Delta t)^2  \Big( \norm{\Psiop}_\infty \norm{(I-P_{\ks} )\charV^l} + \Big( c_1 \norm{T^{-\frac{1}{2}}} \norm{\cB} + \frac{\sqrt{2}}{ \kappa }\norm{\Psiop}_\infty  \Big) \norm{\cB^*T^{-\frac{1}{2}} \charV^l} \Big) \norm{T F^l} \\
                 &\leq  \frac{(\Delta t)^2}{\sigma } 
                 \Big( 
                 \norm{\Psiop}_\infty^2 \hspace{-2pt} \norm{(I-P_{\ks} )\charV^l}^2 
                 + 2 \Big( c_1^2 \norm{T^{-1}} \norm{\cB}^2  + \frac{2}{ \kappa^{2}} \norm{\Psiop}_\infty^2 \Big) \hspace{-2pt}  \norm{\cB^*T^{-\frac{1}{2}} \charV^l}^2 
                 \Big) 
                 +\frac{(\Delta t)^2 \sigma }{2} \norm{T F^l}^2
                 \, ,
                 \end{align*}
                 for all $\sigma > 0$, where we have applied Young's inequality. The bound for $\norm{D^l \charV^l}$ that we used above can be derived similarly as in \eqref{eq:TDVbound} using the definition of $D^l$ in \eqref{eq:defDlFl}. 
                 For the second order change of the potential energy, we have 
                \begin{align*}
                \abs{\RN{7}} &= \abs{\Pot(\charX^{l+1}) + \hat{\Pot}(\charX^{l+1}) - \Pot(\charX^{l}) - \hat{\Pot}(\charX^{l}) + \Delta t \langle F^l, \charV^{l+1} \rangle } \\
                &= \Delta t \int_0^1 \abs{ \langle  (\nabla \Pot + \nabla \hat{\Pot})(\charX^l + s \Delta t \charV^{l+1}) - (\nabla \Pot + \nabla \hat{\Pot} )(\charX^l) , \charV^{l+1} \rangle } \, \mathrm{d}s \, , \\ 
                &
                \leq \frac{3}{2} (\Delta t)^2 L \big( K_{\Psiop}^2 \norm{(I-P_{\ks} )\charV^l}^2 + K_{\cB}^2 \norm{\cB^*T^{-\frac{1}{2}} \charV^l} + (\Delta t)^2 \norm{TF^l}^2    \big) \, .
            \end{align*}
            In the last step, we have used the velocity bound \eqref{eq:semexVlpbound} for $V^{l+1}$ and the Cauchy Schwarz inequality $(a+b + c )^2 \leq 3 (a^2 +  b^2 + c^2)$ for all $a,b,c \in \R$.
            Lastly, we have 
            \begin{align*}
                \RN{8} &= \varepsilon \Delta t \langle T(D^l \charV^l - F^l) , T (F^{l+1} - F^l) \rangle 
                \leq 2\varepsilon \Delta t \norm{T (D^l \charV^l - F^l)} \norm{ F^{l+1} - F^l } \\
                &\leq  2\varepsilon L (\Delta t)^2 \big( \norm{T D^l\charV^l} + \norm{T F^l} \big) 
                \Big( K_{\Psiop} \norm{(I-P_{\ks} )\charV^l} + K_{\cB} \norm{\cB^*T^{-\frac{1}{2}} \charV^l} + \Delta t \norm{TF^l} \Big) \nonumber \\
                &\leq 2\varepsilon L (\Delta t)^2 \Big( 
                2 \norm{\Psiop}_\infty \norm{(I-P_{\ks} )\charV^l} + \sqrt{2} \Big(c_1 \norm{\cB} + \frac{2}{\kappa } \norm{\Psiop}_\infty \Big) \norm{\cB^*T^{-\frac{1}{2}} \charV^l} 
                 + \norm{TF^l} \Big)  \nonumber \\
                &\quad \Big( K_{\Psiop} \norm{(I-P_{\ks} )\charV^l} + K_{\cB} \norm{\cB^*T^{-\frac{1}{2}} \charV^l} + \Delta t \norm{TF^l} \Big)
                \nonumber \\
                &= 2\varepsilon L (\Delta t)^2 \Big[ 
                2 \norm{\Psiop}_\infty K_{\Psiop} \norm{(I-P_{\ks} )\charV^l}^2 
                + \sqrt{2} \Big(c_1 \norm{\cB} + \frac{2}{\kappa } \norm{\Psiop}_\infty \Big) K_{\cB} \norm{\cB^*T^{-\frac{1}{2}} \charV^l }^2 \nonumber \\
                &\quad 
                + \Delta t \norm{TF^l}^2 
                 + 2 \norm{\Psiop}_\infty \norm{(I-P_{\ks} )\charV^l} \Big( 
                K_{\cB} \norm{\cB^*T^{-\frac{1}{2}} \charV^l} + \Delta t \norm{TF^l}
                \Big) \nonumber \\
                &\quad + \sqrt{2} \Big(c_1 \norm{\cB} + \frac{2}{\kappa } \norm{\Psiop}_\infty \Big) \norm{\cB^*T^{-\frac{1}{2}} \charV^l} \Big(
                K_{\Psiop} \norm{(I-P_{\ks} )\charV^l} 
                + \Delta t \norm{TF^l} 
                \Big) \nonumber \\
                &\quad +  \norm{TF^l} \Big( 
                K_{\Psiop} \norm{(I-P_{\ks} )\charV^l}
                + K_{\cB} \norm{\cB^*T^{-\frac{1}{2}} \charV^l}
                \Big)
                \Big] \, .
            \end{align*}
            Next, we apply Young's inequality to each of the last three lines above and obtain 
            \begin{align*}
                 \big\lVert (I&-P_{\ks} )\charV^l \big\rVert \Big( 
                K_{\cB} \norm{\cB^*T^{-\frac{1}{2}} \charV^l} + \Delta t \norm{TF^l}
                \Big) \\
                &\leq  \frac{K_{\cB}}{2} \Big( \norm{(I-P_{\ks} )\charV^l}^2 +  \norm{\cB^*T^{-\frac{1}{2}} \charV^l}^2 \Big)  
                +   \frac{\Delta t}{2}  \Big( \norm{ (I- P_{\ks} )\charV^l }^2 +  \norm{TF^l}^2 \Big) \\
                &= \frac{1}{2} \big(K_{\cB} 
                + \Delta t
                \big) \norm{(I-P_{\ks} )\charV^l}^2
                +  \frac{ K_{\cB}}{2} \norm{\cB^*T^{-\frac{1}{2}} \charV^l}^2 
                 +  \frac{\Delta t}{2}  \norm{TF^l}^2 \, ,
                \\
                 \Big\lVert \cB^* &T^{-\frac{1}{2}} \charV^l \Big\rVert \Big(
                K_{\Psiop} \norm{(I-P_{\ks} )\charV^l} 
                + \Delta t \norm{TF^l} 
                \Big) \\
                &\leq \frac{1}{2} \Big(  K_{\Psiop} \Big( \norm{\cB^*T^{-\frac{1}{2}}\charV^l }^2 +  \norm{(I-P_{\ks} )\charV^l}^2  \Big) 
                +  \Delta t \Big( \norm{\cB^*T^{-\frac{1}{2}}\charV^l}^2 + \norm{TF^l}^2  \Big) \Big) \\
                &= \frac{1}{2} \Big(  K_{\Psiop} \norm{(I-P_{\ks} )\charV^l}^2 
                +  \Big(
                K_{\Psiop} + \Delta t
                \Big) \norm{\cB^*T^{-\frac{1}{2}} \charV^l }^2  
                + \Delta t \norm{TF^l}^2 \Big) \, , \\
                \big\lVert T F^l\big\rVert &\Big( 
                K_{\Psiop} \norm{(I-P_{\ks} )\charV^l}
                + K_{\cB} \norm{\cB^*T^{-\frac{1}{2}} \charV^l}
                \Big) \\
                &\leq \frac{1}{2} K_{\Psiop} \Big( \norm{(I-P_{\ks} )\charV^l}^2 + \norm{TF^l}^2 \Big) 
                + \frac{K_{\cB}}{2} \Big(\norm{TF^l}^2 +  \norm{\cB^*T^{-\frac{1}{2}} \charV^l}^2 \Big) \\ 
                &=  \frac{1}{2} K_{\Psiop} \norm{(I-P_{\ks} )\charV^l}^2
                + \frac{K_{\cB}}{2} \norm{\cB^*T^{-\frac{1}{2}} \charV^l}^2 + \frac{1}{2} \Big( K_{\Psiop} + K_{\cB}  \Big) \norm{TF^l}^2 \, .
            \end{align*}
            Therefore, we obtain 
            \begin{align*}
                \RN{8}
                &\leq 
                2 \varepsilon L (\Delta t)^2 \Big( 
                    \Big[ \Big(2 \norm{\Psiop}_\infty + \frac{\sqrt{2}}{2} \Big(c_1 \norm{\cB} + \frac{2}{\kappa} \norm{\Psiop}_\infty \Big) + \frac{1}{2} \Big) K_{\Psiop} 
                    +  \norm{\Psiop}_\infty (K_{\cB} + \Delta t)
                    \Big] \norm{(I-P_{\ks} )\charV^l}^2 \\ 
                    &\quad 
                    +    \Big[\frac{\sqrt{2}}{2} \Big(c_1 \norm{\cB} + \frac{2}{\kappa} \norm{\Psiop}_\infty \Big) (2 K_{\cB} + K_{\Psiop} + \Delta t)
                +  \frac{K_{\cB}}{2} (2 \norm{\Psiop}_\infty
                 + 1)
                \Big]  \norm{\cB^*T^{-\frac{1}{2}} \charV^l}^2 \\ 
                    &\quad + 
                    \Big[ \Big( 1 + \norm{\Psiop}_\infty + \frac{\sqrt{2}}{2} \Big(c_1 \norm{\cB} + \frac{2}{\kappa} \norm{\Psiop}_\infty \Big) \Big) \Delta t
                 +\frac{1}{2} (K_{\Psiop} + K_{\cB}) 
                 \Big] \norm{T F^l}^2 \Big) \, .
            \end{align*}
            In order to shorten the notation in the following, we introduce the terms 
            \begin{align*}
                \mu_1 &\coloneq  \Delta t \Big( \frac{4 \varepsilon}{\tau \Delta t}  \norm{\Psiop}_\infty^2
                + 4  \norm{\Psiop}_\infty^2
                + \frac{\norm{\Psiop}_\infty^2 }{\sigma} 
                + \frac{3}{2} L K_{\Psiop}^2
                \Big) \hspace{-2pt}  
                + \hspace{-1pt} 2\varepsilon L \Big[ K_{\Psiop} + \frac{K_{\cB}}{2}
                + \frac{\sqrt{2}}{2\kappa}  K_{\Psiop}
                +\frac{\Delta t}{2}
                \Big]  
                 \nonumber \\
                & 
                 + 2 \varepsilon L \Big[ \Big(2 \Delta t \norm{\Psiop}_\infty + \frac{\sqrt{2}}{2} (c_1 \Delta t \norm{\cB} + \frac{2 \Delta t}{\kappa} \norm{\Psiop}_\infty ) + \frac{\Delta t}{2}  \Big) K_{\Psiop} 
                    +  \Delta t \norm{\Psiop}_\infty (K_{\cB} + \Delta t)
                    \Big] \, , \\
                \mu_2 &\coloneq  1 
                    -  4 \norm{\cB}^2 \frac{c_1 \Delta t }{\tau} \frac{\varepsilon}{\Delta t} 
                    -  4 \norm{\cB}^2 c_1 \Delta t
                    -  2\frac{\norm{\cB}^2}{\sigma } \norm{T}^{-1} c_1 \Delta t
                    \, , \\
                \mu_3 &\coloneq 2 \varepsilon \Delta t L \Big[ \frac{\sqrt{2}}{\kappa}  K_{\cB} 
                + \frac{K_{\cB}}{2} 
                + \frac{\sqrt{2}}{2 \kappa } K_{\Psiop}
                + \frac{\sqrt{2} \Delta t}{2 \kappa } 
                \Big] 
                + \frac{4 (\Delta t)^2}{\kappa^{2}} \norm{\Psiop}_\infty^2 \Big( 4 \frac{\varepsilon}{\Delta t \tau} +  \frac{1}{\sigma} + 4  \Big)
                + \frac{3}{2} L K_{\cB}^2 (\Delta t)^2 \\
                &\quad +  2 \varepsilon L \Delta t \Big[\frac{\sqrt{2}}{2} \Big(c_1 \Delta t \norm{\cB} + \frac{2 \Delta t}{\kappa} \norm{\Psiop}_\infty \Big) (2 K_{\cB} + K_{\Psiop} + \Delta t)
                +   \frac{\Delta t K_{\cB}}{2} (2 \norm{\Psiop}_\infty
                 + 1)
                \Big] , \\
                \mu_4 &\coloneq  \frac{2\sqrt{2} \varepsilon L}{\kappa } 
                    + \frac{3}{2} L (\Delta t)^2  
                    + 2 \varepsilon L \Big[  \Delta t + \Delta t \norm{\Psiop}_\infty + \frac{\sqrt{2}}{2} (c_1 \Delta t \norm{\cB} + \frac{2 \Delta t}{\kappa} \norm{\Psiop}_\infty )  
                 +\frac{1}{2} (K_{\Psiop} + K_{\cB}) 
                 \Big] \, .
            \end{align*}
            In summary, by collecting all upper bounds for the terms $\RN{1}, \ldots , \RN{8}$, we obtain
            \begin{align}
                &H_{\text{cl}, \varepsilon}(\charX^{l+1}, \charV^{l+1}) \nonumber \\
                &= H_{\text{cl}, \varepsilon}(\charX^l, \charV^l)
                 - \Delta t 
                \Big\{ \underline{\psi}   - \Delta t \Big( \frac{4 \varepsilon}{\tau \Delta t}  \norm{\Psiop}_\infty^2
                + 4  \norm{\Psiop}_\infty^2
                + \frac{\norm{\Psiop}_\infty^2 }{\sigma} 
                + \frac{3}{2} L K_{\Psiop}^2
                \Big)  \nonumber \\
                &\hspace{3.55pt}
                 - 2 \varepsilon L \Big[ \Big(2 \Delta t \norm{\Psiop}_\infty + \frac{\sqrt{2}}{2} \Big( c_1 \Delta t \norm{\cB} + \frac{2}{\kappa} \Delta t \norm{\Psiop}_\infty \Big) + \frac{\Delta t}{2}  \Big) K_{\Psiop} 
                      +  \Delta t \norm{\Psiop}_\infty (K_{\cB} + \Delta t)
                    \Big]  \nonumber \\
                    &\hspace{3.55pt} - 2\varepsilon L \Big[ K_{\Psiop} + \frac{K_{\cB}}{2}
                + \frac{\sqrt{2}}{2 \kappa }  K_{\Psiop}
                +\frac{\Delta t}{2}
                \Big]  
                    \Big\} \norm{(I-P_{\ks} )\charV^l}^2 \nonumber \\
                &\hspace{3.55pt} -
                    \Big\{ \Big( 1 
                    -  4 \norm{\cB}^2 \frac{c_1 \Delta t }{\tau} \frac{\varepsilon}{\Delta t} 
                    -  4 \norm{\cB}^2 c_1 \Delta t
                    -  2\frac{\norm{\cB}^2}{\sigma } \norm{T}^{-1} c_1 \Delta t
                    \Big) c_1 \Delta t 
                    - \frac{3}{2} L K_{\cB}^2 (\Delta t)^2 \nonumber \\
                &\hspace{3.55pt} - 2 \varepsilon \Delta t L \Big[ \frac{\sqrt{2}}{\kappa} K_{\cB} 
                + \frac{K_{\cB}}{2} 
                + \frac{\sqrt{2}}{2 \kappa } K_{\Psiop}
                + \frac{\sqrt{2}\Delta t}{2\kappa } 
                \Big] 
                 - \frac{4 (\Delta t)^2}{\kappa^2} \norm{\Psiop}_\infty^2 \Big( 4 \frac{\varepsilon}{\Delta t \tau} +  \frac{1}{\sigma} + 4  \Big)
                 \nonumber \\
                &\hspace{3.55pt} -  2 \varepsilon L \Delta t \Big[\frac{\sqrt{2}}{2} \Big(c_1 \Delta t \norm{\cB} + \frac{2 \Delta t}{\kappa } \norm{\Psiop}_\infty \Big) (2 K_{\cB} + K_{\Psiop} + \Delta t) 
                +   \frac{\Delta t K_{\cB}}{2} (2 \norm{\Psiop}_\infty
                 + 1)
                \Big]
                \Big\} \norm{\cB^*T^{-\frac{1}{2}} \charV^l}^2 \nonumber \\
                &\hspace{3.55pt} - 
                (\Delta t)^2 \Big\{ 
                    \Big(1 - \frac{ \tau}{2} \Big) \frac{\varepsilon}{\Delta t}  
                    -  \norm{T^{-1}} 
                    - \frac{ \sigma }{2} 
                    - \frac{2\sqrt{2}}{\kappa} \varepsilon L 
                    - \frac{3}{2} L (\Delta t)^2  \nonumber \\ 
                    &\hspace{3.55pt} - 2 \varepsilon L \Big[  \Delta t + \Delta t \norm{\Psiop}_\infty + \frac{\sqrt{2}}{2} \Big(c_1 \Delta t \norm{\cB} + \frac{2 \Delta t}{\kappa} \norm{\Psiop}_\infty \Big)  
                 +\frac{1}{2} (K_{\Psiop} + K_{\cB}) 
                 \Big]
                \Big\} \norm{TF^l}^2 \nonumber \\
                &= H_{\text{cl}, \varepsilon} (\charX^l, \charV^l) 
                - \Delta t (\underline{\psi} - \mu_1 ) \norm{(I-P_{\ks} )\charV^l}^2
                - ( \mu_2 c_1 \Delta t  - \mu_3 ) \norm{\cB^*T^{-\frac{1}{2}} \charV^l}^2 \nonumber \\
                &\hspace{3.55pt}
                - (\Delta t)^2 \Big( 
                    \Big(1 - \frac{ \tau}{2} \Big) \frac{\varepsilon}{\Delta t}  
                    -  \norm{T^{-1}} 
                    - \frac{ \sigma }{2} - \mu_4 \Big) \norm{TF^l}^2 \, . \label{eq:semexHempsbound}
            \end{align}
            \textbf{Choice of the parameters:} 
            We proceed by finding suitable choices of the parameters $\varepsilon, \Delta t, \rho$ and $c$ to ensure that \eqref{eq:semexHepsbound} holds and $H_{\text{cl}, \varepsilon}$ is bounded from below. 
            First, we notice that if we choose $\alpha = 1/{66}$,  $\sigma = \tau = 1/{2}$ and $\varepsilon = 3 \Delta t$ and if $\rho$ satisfies 
            \begin{align}
                \rho &\leq \frac{\alpha}{ (\Delta t)^3 (1+c (\Delta t)^2 ) \norm{\cB \cB^*}} \, , \label{eq:rhoupbound} \\
               \text{then } c_1 \Delta t \norm{\cB \cB^*} &= c_2 \norm{\cB \cB^*} \leq \alpha 
                \text{ and thus } 
                \norm{T^{-1}} \leq \frac{1}{1- c_2 \norm{\cB \cB^*}} \leq \frac{66}{65} < \frac{5}{4}  \, . \nonumber
            \end{align}
           The norm bound on $T^{-1}$ follows from the Neumann series. Thus, we note that our standing assumption $c_2 \norm{\cB \cB^*} \leq \alpha <  1$ is satisfied. Since $\norm{\cB \cB^*} = \norm{\cB}^2$, we obtain 
           \begin{align}
               \mu_2 &= 1 
                    -  4 \norm{\cB}^2 \frac{c_1 \Delta t }{\tau} \frac{\varepsilon}{\Delta t} 
                    -  4 \norm{\cB}^2 c_1 \Delta t
                    -  2\frac{\norm{\cB}^2}{\sigma } \norm{T}^{-1} c_1 \Delta t 
                    \geq 1 - 33 \alpha = \frac{1}{2}  . \label{eq:mu2lowbound}
           \end{align}
           Similarly, we have
           \begin{align}
                    \Big(1 - \frac{ \tau}{2} \Big) \frac{\varepsilon}{\Delta t}  
                    -  \norm{T^{-1}} 
                    - \frac{ \sigma }{2} &\geq \frac{9}{4} - \frac{5}{4} - \frac{1}{4} = \frac{3}{4} \, . \label{eq:Tfllowbound}
            \end{align}
            Now, fix $\gamma \in (0,1]$, $\overline{c} \geq 0$ and additionally assume that 
            \begin{align}
                0 \leq c \leq \overline{c} \text{ and } \rho \geq \frac{\gamma \alpha}{ (\Delta t)^3 (1+c (\Delta t)^2 ) \norm{\cB \cB^*}} \, . \label{eq:rholowbound}
            \end{align}
            Then for $0 < \Delta t \leq 1$, we have the following upper bounds
            \begin{align}
                K_{\Psiop} &= 1 + 2 \Delta t \norm{\Psiop}_\infty \leq 1 + 2 \norm{\Psiop}_\infty \eqcolon \overline{K}_{\Psiop}, \nonumber \\
                K_{\cB} &= \sqrt{2}\big( \frac{1}{\kappa}
                 +   c_1 \Delta t \norm{\cB} + \frac{2 \Delta t}{\kappa} \norm{\Psiop}_\infty  \big)
                \leq  \sqrt{2} \Big( \frac{1}{\kappa} + \frac{\alpha}{\norm{\cB}} + \frac{2}{\kappa} \norm{\Psiop}_\infty \Big) 
                \eqqcolon \overline{K}_{\cB}, \nonumber \\
                c_3 \norm{\cB}^2 &= c_3 \norm{\cB \cB^*} = c \rho (\Delta t)^3 \norm{\cB \cB^*} \leq \frac{c \alpha}{1 + c (\Delta t)^2} \leq \overline{c} \alpha \, , \label{eq:c3upbd} \\
                L &= \begin{cases}
                     L_{\nabla \Pot} + c_3 \norm{T^{-1}} \norm{\cB \cB^*}, \, &\hat{\Pot} \text{ is given by } \eqref{eq:semiexnonstatad} \, , \\ 
                    L_{\nabla \Pot} + c_3 \norm{T^{-1}} \norm{\cB \cB^*} + \norm{ \nabla \Pot (\Xeq) }, \, &\hat{\Pot} \text{ is given by } \eqref{eq:semiexstatad} \, ,
                \end{cases} \nonumber \\
                &\leq L_{\nabla \Pot} + \frac{5}{4} \overline{c} \alpha
                +
                \begin{cases}
                0, \, &\hat{\Pot} \text{ is given by } \eqref{eq:semiexnonstatad} \, , \\
                \norm{ \nabla \Pot (\Xeq) }, \, &\hat{\Pot} \text{ is given by } \eqref{eq:semiexstatad} \, ,
                \end{cases}
                \eqqcolon \overline{L} \, . \nonumber
                \end{align}
             Let us assume that \eqref{eq:Psioplowbd} holds and that $\hat{\Pot}$ is given by \eqref{eq:semiexnonstatad}. Moreover, let $C > 0, \mathcal{\underline{V}}$ be as in \eqref{eq:potgrwthbd}. Then there exists $ 0 < h \leq 1$ depending on $\gamma, \overline{K}_\psi,  \overline{K}_{\cB}, \norm{\Psiop}_\infty, \overline{L}, \kappa, \norm{\cB}$ and $\underline{\psi}$ such that for all $0 < \Delta t \leq h$, we have 
            \begin{align}
                \mu_1 &\leq \frac{1}{2} \underline{\psi} , \  
                \mu_3 \leq \frac{\gamma \alpha}{4 \norm{\cB \cB^*}}, \ 
                \mu_4 \leq \frac{1}{4}, \nonumber \\
                \text{ and } \Delta t \Big( 1 + \frac{\overline{c}\alpha }{\norm{\cB}} \Big) &\leq \frac{1}{24} , \ 
                \Delta t \frac{\norm{\cB}}{\sqrt{2}} 
                \leq \frac{1}{12} \, , \ 3 \Delta t C \leq \frac{1}{2} , \, 
                3 \Delta t \sqrt{2} \leq \frac{1}{8} 
                \, .\label{eq:dtHepslowbound}
            \end{align}
            Due to the lower bound \eqref{eq:rholowbound}, we have  
            \begin{align}
                \frac{1}{4} c_1 \Delta t = \frac{1}{4} \rho (\Delta t)^3 (1+c (\Delta t)^2) \geq \frac{\gamma \alpha}{ 4\norm{\cB \cB^*}} \geq \mu_3 \, . \label{eq:muc1dtest}
            \end{align}
            Finally, combining the estimates above with \eqref{eq:semexHempsbound} yields 
            \begin{align}
                H_{\text{cl}, 3\Delta t}(\charX^{l+1}, \charV^{l+1})
                &\leq  H_{\text{cl}, 3\Delta t}(\charX^{l}, \charV^{l})
                - \Delta t (\underline{\psi} - \mu_1 ) \norm{(I-P_{\ks} )\charV^l}^2
                - ( \mu_2 c_1 \Delta t  - \mu_3 ) \norm{\cB^*T^{-\frac{1}{2}} \charV^l}^2 \nonumber \\
                &\quad 
                - (\Delta t)^2 \Big( 
                    \Big(1 - \frac{ \tau}{2} \Big) \frac{\varepsilon}{\Delta t}  
                    -  \norm{T^{-1}} 
                    - \frac{ \sigma }{2} - \mu_4 \Big) \norm{TF^l}^2 \nonumber \\ 
                &\leq  H_{\text{cl}, 3\Delta t}(\charX^{l}, \charV^{l}) 
                - \frac{\Delta t}{2} \underline{\psi}\norm{(I-P_{\ks} )\charV^l}^2
                - \frac{c_1 \Delta t }{4} \norm{\cB^*T^{-\frac{1}{2}} \charV^l}^2
                - \frac{(\Delta t)^2}{2} \norm{TF^l}^2 \, . \label{eq:Hepsgenest}
            \end{align}
            \textbf{Proof of Assertion \ref{ass:Hepsdecr}}: Since we have already established \eqref{eq:Hepsgenest}, it remains to show that $H_{\text{cl}, 3\Delta t}$ is bounded from below. For this purpose, we will use the bounds \eqref{eq:dtHepslowbound} and the upper bound for $\rho$ \eqref{eq:rhoupbound}. Due to \eqref{eq:c3upbd}, we have 
            \begin{align}
                6 \Delta t \Big( 1 + \frac{c_3 \norm{\cB}}{\sqrt{2}}  \Big) &\leq 6 \Delta t \Big( 1 + \overline{c} \frac{\alpha }{\norm{\cB}}  \Big) \leq \frac{1}{4} \, . \label{eq:prefactbound}
            \end{align}
            Recall that $C > 0, \underline{ \mathcal{V}} \in \R$ are given in \eqref{eq:potgrwthbd}. Thus for $Y, W \in \HS$, it holds that
            \begin{align}
                &3\Delta t \langle W, T (\nabla \Pot(Y) + \nabla \hat{\Pot}(Y) ) \rangle
                = 3 \Delta t \langle W, T (\nabla \Pot (Y) + c_3 T^{-\frac{1}{2}} \cB \cB^*T^{-\frac{1}{2} } (Y - \Xeq) ) \rangle \nonumber \\
                &\geq - 3 \Delta t\Big(  \norm{W} 2 \norm{\nabla \Pot(Y) } - \norm{W} \sqrt{2} c_3 \norm{\cB}  \norm{\cB^*T^{-\frac{1}{2}} (Y - \Xeq)} \Big) \nonumber \\
                &\geq 
                -  3 \Delta t \Big( \norm{W}^2 -  C(1 + \Pot(Y) - \underline{\mathcal{V}})  
                - \sqrt{2} c_3 \norm{\cB} \Big( 
                \frac{1}{2}\norm{\cB^*T^{-\frac{1}{2}}(Y - \Xeq) }^2 + \frac{1}{2} \norm{W}^2    \Big) \Big) \nonumber \\
                &= - 3\Delta t \Big( 1 + \frac{c_3 \norm{\cB}}{\sqrt{2}}  \Big) \norm{W}^2 
                - 3 \Delta t C(1 + \Pot(Y) - \underline{\mathcal{V}}) 
                - 3 \Delta t \frac{c_3 \norm{\cB}}{\sqrt{2}} \norm{\cB^*T^{-\frac{1}{2}} (Y - \Xeq)}^2 \nonumber \\
                &\geq 
                - 6\Delta t  \Big( 1 + \frac{c_3 \norm{\cB}}{\sqrt{2}}  \Big) \norm{T^{-\frac{1}{2}} W}^2 - 3 \Delta t C(1 + \Pot(Y) - \underline{\mathcal{V}}) 
                - 3 \Delta t \frac{c_3 \norm{\cB}}{\sqrt{2}} \norm{\cB^*T^{-\frac{1}{2}} (Y - \Xeq)}^2 \nonumber \\
                  &\geq - \frac{1}{4} \norm{T^{-\frac{1}{2}} W}^2 - \frac{1}{2} (1 + \Pot(Y) - \underline{\mathcal{V}}) 
                - \frac{c_3}{4} \norm{\cB^*T^{-\frac{1}{2}} (Y - \Xeq)}^2 \, . \nonumber 
            \end{align}
            In the last step, we have used \eqref{eq:dtHepslowbound} and \eqref{eq:prefactbound}. Recall that $\Pot$ is bounded from below by the constant $\underline{\mathcal{V}}$ by Assumption \ref{ass:control}. Using the estimate above, we obtain
            \begin{align*}
               H_{\text{cl}, 3 \Delta t}(Y,W)  &= \frac{1}{2} \norm{T^{-\frac{1}{2}} W}^2 + \Pot(Y) + \frac{c_3}{2} \norm{\cB^*T^{-\frac{1}{2}} (Y - \Xeq)}^2 
               + 3 \Delta t \langle W, T (\nabla \Pot(Y) + \nabla \hat{\Pot}(Y) \rangle \\
               &\geq \frac{1}{4} \norm{T^{-\frac{1}{2}} W}^2 + \frac{c_3}{4} \norm{\cB^*T^{-\frac{1}{2}} (Y - \Xeq)}^2 + \underline{\mathcal{V}} - \frac{1}{2}   \, ,
            \end{align*}
            which shows that $H_{\text{cl}, 3 \Delta t}$ is bounded from below.
            \\
            \textbf{Proof of Assertion \ref{ass:HepsdecrBsurj}}: To complete the proof, it remains to consider the case where $\hat{\Pot}$ is given by \eqref{eq:semiexstatad}, $\cB$ is surjective, $\ks = \HS$ and $0 < c \leq \overline{c}$. In particular, we note that $P_{\ks} = I$. Thus by using \eqref{eq:hBmeanlowbound}, \eqref{eq:Hepsgenest} reduces to
            \begin{align*}
                H_{\text{cl}, 3\Delta t}(\charX^{l+1}, \charV^{l+1})
                &\leq  
                H_{\text{cl}, 3\Delta t}(\charX^{l}, \charV^{l}) 
                - \frac{c_1 \Delta t \kappa^2}{8} \norm{ \charV^l}^2
                - \frac{(\Delta t)^2}{2} \norm{TF^l}^2 \, . 
            \end{align*}
            In particular, we have shown \eqref{eq:semexHepsboundsurj} with $\nu = \kappa^2 /{8}$.
            It remains to prove that $H_{\text{cl}, 3 \Delta t}$ is bounded from below. Similarly as for the potential \eqref{eq:semiexnonstatad}, we have for $Y, W \in \HS$
            \begin{align*}
                &3\Delta t \langle W, T (\nabla \Pot(Y) + \nabla \hat{\Pot}(Y) ) \rangle \\
                &= 3 \Delta t \Big( 
                \langle W, T (\nabla \Pot(Y)  + 
                c_3 T^{-\frac{1}{2}} \cB \cB^*T^{-\frac{1}{2} } (Y - \Xeq) )
                \rangle
                 - 
                 \langle W, T \nabla \Pot(\Xeq) \rangle 
                 \Big) \\
                &\geq - \frac{1}{4} \norm{T^{-\frac{1}{2}} W}^2 - \frac{1}{2} (1 + \Pot(Y) - \underline{\mathcal{V}}) 
                - \frac{c_3}{4} \norm{\cB^*T^{-\frac{1}{2}} (Y - \Xeq)}^2 
                - 6 \sqrt{2} \Delta t \norm{T^{-\frac{1}{2}} W} \big\lVert \nabla \Pot (\Xeq) \big\rVert
                \\
                &\geq -  \frac{3}{8} \norm{T^{-\frac{1}{2}} W}^2 - \frac{1}{2} (1 + \Pot(Y) - \underline{\mathcal{V}}) 
                - \frac{c_3}{4} \norm{\cB^*T^{-\frac{1}{2}} (Y - \Xeq)}^2 
                \nonumber 
                - \frac{1}{8} \big\lVert \nabla \Pot (\Xeq) \big\rVert^2 \, .
            \end{align*}
            Therefore, lower-boundedness of $H_{\text{cl}, 3 \Delta t }$ follows from the inequality 
            \begin{align}
               H_{\text{cl}, 3 \Delta t}(Y,W)  &= \frac{1}{2} \norm{T^{-\frac{1}{2}} W}^2 + \Pot(Y) - \langle \nabla \Pot (\Xeq), Y - \Xeq \rangle  + \frac{c_3}{2} \norm{\cB^*T^{-\frac{1}{2}} (Y - \Xeq)}^2 \nonumber \\
               &\quad 
               + 3 \Delta t \langle W, T (\nabla \Pot(Y) + \nabla \hat{\Pot}(Y) \rangle \nonumber \\
               &\geq \frac{1}{8} \norm{T^{-\frac{1}{2}} W}^2 + \frac{c_3 \kappa^2 }{8} \norm{ Y - \Xeq }^2 + \underline{\mathcal{V}} - \frac{1}{2}
               - \frac{1}{8} \norm{ \nabla \Pot (\Xeq) }^2
               - \norm{ \nabla \Pot (\Xeq) } \norm{ Y - \Xeq } \nonumber \\ 
               &\geq 
               \frac{1}{8} \norm{T^{-\frac{1}{2}} W}^2 +   \frac{c_3 \kappa^2}{16} \norm{ Y - \Xeq }^2 + \underline{\mathcal{V}} - \frac{1}{2}
               - \frac{1}{8} \norm{ \nabla \Pot (\Xeq) }^2
               - 
               \frac{4}{ c_3 \kappa^2} \norm{ \nabla \Pot (\Xeq) }^2 \nonumber \\ 
               &\geq \frac{1}{8} \norm{T^{-\frac{1}{2}} W}^2 +  \frac{c_3 \kappa^2}{16} \norm{ Y - \Xeq }^2 + \underline{\mathcal{V}} - \frac{1}{2}
               - \Big( \frac{1}{8} + 
                \frac{4}{ c \gamma \alpha \kappa^2} \norm{\cB \cB^*} (1+\overline{c})  \Big) \norm{ \nabla \Pot (\Xeq) }^2 \, . \nonumber
            \end{align}
            In the third line above, we have used \eqref{eq:hBmeanlowbound}, while the last step follows from  
            \begin{align*}
                c_3 = c \rho (\Delta t)^3 \geq c \frac{\gamma \alpha}{(1+c (\Delta t)^2)\norm{\cB \cB^*} } \geq c \frac{\gamma \alpha}{(1 + \overline{c} ) \norm{\cB \cB^*}} \, .
            \end{align*}
            In summary, the lower bound above demonstrates that $H_{\text{cl}, 3 \Delta t }$ is bounded from below.
\end{proof}
\begin{Rem}
            The reason we introduced the modified Hamiltonian $H_{\text{cl}, \varepsilon}$ instead of working directly with $H_{\text{cl}}$ are the terms \RN{5}, \RN{6} and \RN{7} in the energy balance \eqref{eq:semexHepsenbalance}. Even though those terms are of second order in $\Delta t$, they depend on $F^l$, while the dissipative term $\RN{1} = - \Delta t \langle \charV^l, D^l \charV^l \rangle$ is independent of $F^l$. By modifying the Hamiltonian to $H_{\text{cl}, \varepsilon }$, we get the additional dissipative term $ \RN{2} = - \varepsilon \Delta t \norm{TF^l}^2$, which, if the parameters are chosen accordingly, is sufficient to compensate for all second order terms containing $F^l$. The fact that the continuous time energy balance is not preserved on the discrete level is a consequence of the chosen time integration scheme, i.~e.~the semi-explicit Euler scheme. However, since the potential $\Pot$ is in general not quadratic, using an explicit higher order scheme does not resolve the issue of higher order correction terms containing $F^l$. 
    \end{Rem}

    \section*{Acknowledgments}
  This work was funded by the Deutsche Forschungsgemeinschaft (DFG, German Research Foundation) – Project-ID 531152215 – CRC 1701. We acknowledge the assistance of ChatGPT-5 for language editing and stylistic improvements. The authors assume responsibility for all content.

\bibliographystyle{unsrtnat}
\bibliography{references}

\end{document}